\documentclass[UTF-8,reqno]{amsart}

\usepackage{amsmath, amsfonts, amsthm, amssymb, stmaryrd, color, cite}

\usepackage{hyperref}
\newtheorem{theorem}{Theorem}[section]
\newtheorem{lemma}{Lemma}[section]
\newtheorem{proposition}{Proposition}[section]

\theoremstyle{definition}
\newtheorem{definition}{Definition}[section]

\theoremstyle{remark}
\newtheorem{remark}{Remark}[section]

\numberwithin{equation}{section}
\allowdisplaybreaks

\def\f{\frac}

\def\hf1{^\f{1}{1-\xi^2}}

\def\be{\begin{equation}}
\def\en{\end{equation}}
\def\bs{\begin{split}}
\def\es{\end{split}}
\def\ba{\begin{align}}
\def\ea{\end{align}}

\title[Large deviations of invariant measure for 3D stochastic Navier-Stokes equations]
{Large deviations of invariant measure for the 3D stochastic hyperdissipative Navier-Stokes equations}

\author[Z. Qiu]{Zhaoyang Qiu}
\address{School of Applied Mathematics, Nanjing University of Finance and Economics, Nanjing, 210046, China.}
\email{zhqmath@163.com}

\author[H. Liu]{Hui liu}
\address{School of Mathematical Sciences,
Qufu Normal University, Qufu, Shandong, 273155, China.}
\email{liuhuinanshi@qfnu.edu.cn}

\author[ C. Sun]{Chengfeng Sun}
\address{School of Applied Mathematics, Nanjing University of Finance and Economics, Nanjing, 210046, China.}
\email{chengfengsunnufe@163.com}


\date{\today}

\begin{document}
\begin{abstract}
In this paper, we consider the large deviations of invariant measure for the 3D stochastic hyperdissipative Navier-Stokes equations driven by additive noise. The unique ergodicity of invariant measure as a preliminary result is proved using a deterministic argument by the exponential moment and exponential stability estimates. Then, the uniform large deviations is established by the uniform contraction principle. Finally, using the unique ergodicity and the uniform large deviations results, we prove the large deviations of invariant measure by verifying the Freidlin-Wentzell large deviations upper and lower bounds.
\end{abstract}

\maketitle
\section{Introduction}

Dynamical systems influenced by a random fluctuation contribute to the uncertainty in modelling the fluid systems, meanwhile the uncertainty and randomness have a far-reaching impact for the evolving of the dynamical phenomenon, especially in geophysical fluid, climate dynamics, etc. Therefore, randomness must be taken into account, considering the effect on the evolution of long time dynamical behaviour of systems. One of the important schemes leading to the study of various dynamical behaviour for random processes is dynamical systems subject to the effect of random noise. Asymptotics in the theory of random processes include results of the types of both the deviation principle and the long time statistics. More precisely, it is often natural to consider the asymptotic relationship between the distributions of solution to the stochastic system and the  deterministic analogy, that is, studying the limit of small random perturbations. We call it the deviation principle arising from the quantum mechanics, see \cite{Fre} for more physical backgrounds.

In this paper, we consider the large deviations of invariant measure for the 3D stochastic hyperdissipative Navier-Stokes equations driven by additive noise:
\begin{eqnarray}\label{Equ1.1}
\left\{\begin{array}{ll}
\!\!\!\partial_t \mathbf{u}+\upsilon(-\Delta)^{\alpha} \mathbf{u}+(\mathbf{u}\cdot \nabla)\mathbf{u}+\nabla p=\sqrt{\varepsilon}G\frac{dW}{dt},\\
\!\!\!\nabla\cdot \mathbf{u}=0,\\
\end{array}\right.
\end{eqnarray}
where $\mathbf{u}$ is velocity field, $p$ is pressure, $W$ is a cylindrical Wiener process, $G$ is a noise intensity operator defined later, $\varepsilon$ is the small perturbation parameter. The coefficient $\upsilon$ is the viscosity of the fluid, we assume that $\upsilon\equiv1$ here.

We prescribe the initial data
$$\mathbf{u}(0)=\mathbf{u}_0,$$
and the periodic boundary
$$\mathbb{T}^3=[0,2\pi]^3.$$
Then, a fractional power of the Laplace transform, $(-\Delta)^{\alpha}$, is defined through the Fourier transform
$$\widehat{(-\Delta)^{\alpha}} \mathbf{u}=|k|^{2\alpha}\hat{\mathbf{u}},$$
for the meanings of notations $k, \hat{\mathbf{u}}$, see section 2. Dissipation corresponding to a fractional power of Laplacian arises
from modeling real physical phenomena, our motivation for considering equations \eqref{Equ1.1} is mainly mathematical and the goal is to understand that the parameter brings the regularity effect, making the study on the long-time dynamical behaviour possible. Since when $\alpha=1$, equations \eqref{Equ1.1} reduce to the usual stochastic Navier-Stokes equations. It is well-known that the Leray weak solution of the 3D Navier-Stokes equations is non-uniqueness for both deterministic and stochastic cases, see \cite{bv, mart}. A remarkable feature of dynamical system with uniqueness is the memoryless property,
thus, the Markov property. Therefore, for the usual 3D stochastic Navier-Stokes equations, it seems that it is really challenging to study the long-time statistics of transition semigroup due to lack of the Markov property. Therefore, we consider the hyperdissipative critical case, thus, ${\bf\alpha=\frac{5}{4}}$.

We next review the research progress on the stochastic Navier-Stokes equations. The pioneering work \cite{B1} proved the existence of a unique weak solution  of the system under the influence of additive noise. Then, for the multiplicative noise case, Flandoli \cite{F1} proved the existence and uniqueness of martingale and stationary solution. Afterwards, abundant outstanding results came to the force, see \cite{bre,cww,ww} and  the references therein.
The invariant measure and ergodicity were studied in \cite{F2} for 2D case. After that, Da Prato and Debussche \cite{Da2} extended the ergodicity to the 3D case. See also \cite{brze3, Hai-Mat,F2,huang} for more results. Deviation principle as another important research project was also investigated for the stochastic Navier-Stokes equations in many publications. Papers \cite{Bud, Dup} developed the weak convergence method which transforms the proof of deviation principle into the basic qualitative property argument, simplifying the process of proof. Then, the weak convergence method as a powerful tool is widely used for the research on deviation principle of many fluid dynamical models, see \cite{Millet,Chueshov,sun, D1,Duan,WZZ, XZ, ZZ, R1} for more results of the application in the Navier-Stokes equations and other related fluid dynamical models.

The study of the large deviations of invariant measure could trace back to thirty years ago, Sowers \cite{sow} established the large deviations of invariant measure for the reaction-diffusion equation with non-Gaussian perturbations.  Then, the result was generalized to the reaction-diffusion equation with multiplicative noise and non-Lipschitz reaction term by R\"{o}ckner and Cerrai \cite{cr}. Following the strategy of these papers, Brze\'{z}niak et.al\cite{bc} proved the large deviations of invariant measure for 2D stochastic Navier-Stokes equations on the torus in $L^2$ with rate function defined by
$$U(x)=\inf\left\{I(\mathbf{u}):\mathbf{ u}\mathbf{}\in C([0,T];H), \mathbf{u}(0)=0, \mathbf{u}(T)=x\right\},$$
for more details of the quasi-potential, see \cite{bcm}.
After that, Cerrai and Paskal \cite{cp} generalized the result to the 2D Navier-Stokes equations with vanishing noise correlation on a torus. All these results were only related to the case of trivial limiting dynamics, thus, the global attractor of the limiting equations is a singleton. Furthermore, Martirosyan\cite{mar} considered the limit equations in a bounded domain with the Dirichlet boundary condition having arbitrary finite number of equilibrium, proved that the family of invariant measure was exponential tightness which implies the large deviations upper bound, and when the limit set of equations is singleton, the large deviations lower bound follows automatically.

Our goal of this paper is to prove that the family of invariant measures of the 3D stochastic hyperdissipative Navier-Stokes equations satisfies the large deviations in $L^2$ with rate function $U(x)$.

The basic qualitative properties: existence and uniqueness of the 3D hyperdissipative Naiver-Stokes equations including the stochastic version and the deterministic skeleton equations, as the preliminary results, are proved at the beginning. Unlike the 2D case on a torus, the orthogonality of $A\mathbf{u}$ and $B(\mathbf{u})$ in $L^2$ which is deeply dependent on \cite{bc,cp} does not hold in 3D case. By exploring the advantage of the hyperdissipative construction, we could control the nonlinear term and obtain the high-order regularity estimate with time-weight in $H^1$.

With the existence and uniqueness in hands, we next prove the existence of a unique ergodic invariant measure. For the existence of invariant measure, we use a generalized Krylov-Bogoliubov method introduced by Maslowski-Seidler that relaxes the Feller condition to sequentially weak Feller, which simplifies the proof. Therefore, we only need the weak compactness and regularity estimate of $H$ itself is adequate. Then, we prove the unique ergodicity of invariant measure by a deterministic argument, which relies on the exponential moment and exponential stability estimates. Here, since the noise has a small perturbation coefficient, we do not require that the viscosity of fluid is large. Therefore, it is reasonable to assume $\upsilon\equiv 1$.

Another important factor is to obtain the large deviations of the family of solutions, uniformly with respect to initial data $\mathbf{u}_0$ in a bounded set. Using the idea of \cite{cp,sbd}, we split the equations into a linear stochastic equations and a random Navier-Stokes equations. For the random 3D Navier-Stokes equations, we can only show that the solution is local Lipschitz continuous of the solution of the linear equations in a bounded set $H^\frac{5}{4}$. Following the uniform contraction principle, the large deviations in $C([0,T]; H^\frac{5}{4})$ and the exponential tightness of solutions of the linear equations are required.
In order to establish the large deviations of the distribution of solutions to linear equations, we introduce a new operator $\Gamma$ to deal with the control term, and prove that operator $\Gamma$ is compact of the topology $C([0,T]; H^\frac{5}{4})$ with respect to the $L^2$-weak topology. Then, combining the weak convergence method and the compactness of $\Gamma$, we show the compactness of the distribution of solutions. To prove the exponential tightness, we will design a appropriate Lyapunov function, then use some tricks to control the extra terms arising from martingale part, see Lemma 5.2.

Finally, using the ergodicity of invariance measure and the uniform large deviations results, we can verify the Freidlin-Wentzell large deviations upper and lower bounds by the arguments developed by \cite{sow}, obtain the large deviations of invariant measure in $L^2$.

We arrange the rest of this paper as follows. In section 2, we introduce all the functional settings
and the well-posedness results for the equations \eqref{Equ1.1} and the corresponding skeleton equations. In section 3, we prove the ergodicity of invariance measure by combining the Maslowski-Seidler theory with a deterministic argument. Applying the weak convergence method, we establish the large deviations in section 4 for the linear equations. The uniform large deviations is proved in section 5 by uniform contraction principle. The main result of the large deviations of invariant measure is obtained in section 6.

Throughout the paper, universal constants depending only on the dimension, or initial data and other parameters are denoted by $C$, etc. which may be different at each occurrence.

\section{Preliminaries and well-posedness}

\subsection{Preliminaries}

In order to formulate the Navier-Stokes equations, we introduce the standard space theory. On the torus $\mathbb{T}^3$, for the $L^2$ functional space, it is convenient to work with the Fourier expansion of $\mathbf{u}$,
\begin{align*}\mathbf{u}=\sum_{k\in \mathbb{Z}^3}\hat{\mathbf{u}}_ke^{2\pi {\rm i} k\cdot x}, ~\overline{\hat{\mathbf{u}}_{-k}}=\mathbf{u}_k,\end{align*}
with norm
$$\|\mathbf{u}\|_{L^2}^2=\sum_{k\in \mathbb{Z}^3}|\hat{\mathbf{u}}_k|^2.$$
However, in other $L^p, p>2$ spaces the situation is somewhat more complicated, we refer to \cite[Chapter 1.5]{RR} for more details.

For any integer $s> 0$, $H^{s}$ denotes the Sobolev space with functions $\partial_x^\alpha\mathbf{u} \in L^2$ for positive integer $0<\alpha\leq s$, endowed with the norm
$$\|\mathbf{u}\|_{H^{s}}^2=\sum_{k\in \mathbb{Z}^3_0}|k|^{2s}|\hat{\mathbf{u}}_k|^2.$$

We use the operator
$$\widehat{(\Lambda^\alpha \mathbf{u})}=|k|^{\alpha}\hat{\mathbf{u}},$$
for any $\alpha> 0$, $\hat{\mathbf{u}}$ is the Fourier transform of $\mathbf{u}$. Then, we have $\|\Lambda^\alpha \cdot\|_{L^2}$ is equivalent to $\|\cdot\|_{H^\alpha}$.
We denote by $H^{-\alpha}$ the dual of $H^{\alpha}$ for any $\alpha>0$, with norm
$$\|\mathbf{u}\|_{H^{-\alpha}}^2=\sum_{k\in \mathbb{Z}_0^3}|k|^{-2\alpha}|\hat{\mathbf{u}}_k|^2.$$

 Denote
$$H:=\left\{\mathbf{u}\in L^2: {\rm div} \mathbf{u}=0, \int_{\mathbb{T}^3}\mathbf{u}dx=0\right\}$$
with the norm
$$\|\mathbf{u}\|^2_{H}=\|\mathbf{u}\|^2_{L^2}=(\mathbf{u},\mathbf{u}),$$
where $(\cdot, \cdot)$ means the inner product of $L^2$. Let $H_w$ be the Hilbert space $H$ with the weak topology. 

Introduce
$$C([0,T];H_w):={\rm the~ space ~of}~ H ~{\rm valued ~weakly ~continuous ~function},$$
endowed with the weak topology such that the mapping
$$\mathbf{u}\mapsto \langle \mathbf{u},h\rangle$$
is continuous for any $h\in H$, which is a quasi-Polish space.  Denote
$$V:=\left\{\mathbf{u}\in H: \frac{\partial \mathbf{u}}{\partial x}\in L^2\right\}$$
with the norm
$$\|\mathbf{u}\|^2_{V}=\|\nabla \mathbf{u}\|^2_{L^2}=\sum_{i,j=1}^3\int_{\mathbb{T}^3}\left|\frac{\partial \mathbf{u}_j}{\partial x_i}\right|^2dx.$$
The operator $A$ is defined by
$$A\mathbf{u}:=-P\Delta \mathbf{u},~ \mathbf{u}\in D(A)=H^{2}\cap V,$$
where $P$ is the Helmhotz–Hodge projection operator from $L^2$ into $H$. Define by the operator $A^\alpha$
$$A^\alpha\mathbf{u}:=P(-\Delta)^\alpha \mathbf{u}, ~\mathbf{u}\in D(A^\alpha)=H^{2\alpha}.$$
More details on $A^\alpha$  can be found in Chapter 5 of Stein's book \cite{ste}.

Denote by $V'$ the dual of $V$, then we have
$$V\subseteq H \subseteq V',$$
in this way, we could consider $A$ as a bounded operator from $V$ into $V'$. Applying the theory of symmetric, compact operator for $A^{-1}$, one can prove the existence of an orthonormal basis $\{e_{k}\}_{k\geq 1}$ for $H$ of eigenfunctions of $A$, and a sequence of positive eigenvalues $\lambda_k=|k|^2$ with $\lambda_k\nearrow\infty$ as $k\rightarrow\infty$, that is,
$$Ae_k=|k|^2e_k,~{\rm for} ~k\in \mathbb{Z}_0^3.$$
For $\mathbf{u}\in H^\beta$, we have the Poincar\'{e} inequality
$$\|\Lambda^\alpha\mathbf{u}\|_{H}^2\leq \|\Lambda^\beta \mathbf{u}\|_{H}^2,~ {\rm if}~ \alpha\leq \beta. $$

Define the bilinear map
$$B(\mathbf{u},\mathbf{v})=P(\mathbf{u}\cdot\nabla) \mathbf{v},$$
and define the tri-linear map $b(\cdot, \cdot, \cdot): V\times V\times V\rightarrow \mathbb{R}$ by
$$b(\mathbf{u},\mathbf{v},\mathbf{w})=\int_{\mathbb{T}^3}\mathbf{u}\cdot\nabla \mathbf{v}\cdot \mathbf{w}dx, ~ \mathbf{u},\mathbf{v},\mathbf{w}\in V,$$
then,
$$(B(\mathbf{u},\mathbf{v}), \mathbf{w})=b(\mathbf{u},\mathbf{v},\mathbf{w}).$$
Integrating by parts,
$$(B(\mathbf{u},\mathbf{v}), \mathbf{w})=-(B(\mathbf{u},\mathbf{w}), \mathbf{v}), ~\mathbf{u}, \mathbf{v}, \mathbf{w}\in V$$
which implies
\begin{align}\label{2.1}b(\mathbf{u},\mathbf{v},\mathbf{v})=0, ~\mathbf{u},\mathbf{v} \in V.\end{align}

After an application of operator $P$ to equations \eqref{Equ1.1}, we could rewrite it by abstract form as an evolutionary dynamical system
\begin{equation}\label{e1*}
d\mathbf{u}+A^\frac{5}{4}\mathbf{u}dt+B(\mathbf{u},\mathbf{u})dt=\sqrt{\varepsilon}GdW,
\end{equation}
with initial data $\mathbf{u}(0)=\mathbf{u}_0$.

Let $\mathcal{S}:=(\Omega,\mathcal{F},\{\mathcal{F}_{t}\}_{t\geq0},\mathrm{P}, W)$ be a fixed stochastic basis and $(\Omega,\mathcal{F},\mathrm{P})$ a complete probability space. $\{\mathcal{F}_{t}\}_{t\geq0}$ is a filtration satisfying all usual conditions.
Denote by $L^p(\Omega; L^q(0,T;X)), p\in [1,\infty), q\in [1, \infty]$ the space of processes with values in $X$ defined on $\Omega\times [0,T]$ such that

i. $\mathbf{u}$ is measurable with respect to $(\omega, t)$, and for each $t$, $\mathbf{u}(t)$ is $\mathcal{F}_{t}$-measurable;

ii. For almost all $(\omega, t)$, $\mathbf{u}\in X$ and
\begin{align*}
\|\mathbf{u}\|^p_{L^p(\Omega; L^q(0,T;X))}\!=\!
\begin{cases}
\mathrm{E}\left(\int_{0}^{T}\|\mathbf{u}\|_{X}^qdt\right)^\frac{p}{q},~& {\rm if}~q\in [1,\infty),\\
\mathrm{E}\left(\sup_{t\in[0,T]}\|\mathbf{u}\|^p_X\right),~ & {\rm if}~q=\infty.
\end{cases}
\end{align*}
Here, $\mathrm{E}$ denotes the mathematical expectation.

Assume that $Q$ is a linear positive operator on the Hilbert space $H$, which is trace and hence compact. Let $W$ be a Wiener process defined on the Hilbert space $H$ with covariance operator $Q$, which is adapted to the complete, right continuous filtration $\{\mathcal{F}_{t}\}_{t\geq 0}$. Let $\{\mathbf{e}_{k}\}_{k\geq 1}$ be a complete orthonormal basis of $H$ such that $Q\mathbf{e}_{i}=\lambda_{i}\mathbf{e}_{i}$, then $W$ can be written formally as the expansion $W(t,\omega)=\sum_{k\geq 1}\sqrt{\lambda_{k}}\mathbf{e}_{k}W_k(t,\omega)$, where $\{W_{k}\}$ is a sequence independent standard 1-D Brownian motions, see \cite{Zabczyk} for more details.

Let $H_{0}=Q^{\frac{1}{2}}H$, then $H_{0}$ is a Hilbert space with the inner product
\begin{equation*}
\langle h,g\rangle_{H_{0}}=\langle Q^{-\frac{1}{2}}h,Q^{-\frac{1}{2}}g\rangle_{H},~ \forall~~ h,g\in H_{0},
\end{equation*}
with the induced norm $\|\cdot\|_{H_{0}}^{2}=\langle\cdot,\cdot\rangle_{H_{0}}$. The imbedding map $i:H_{0}\rightarrow H$ is Hilbert-Schmidt and hence compact operator with $ii^{\ast}=Q$. Now considering another separable Hilbert space $X$ and let $L_{Q}(H_{0},X)$ be the space of linear operators $S:H_{0}\rightarrow X$ such that $SQ^{\frac{1}{2}}$ is a linear Hilbert-Schmidt
operator from $H$ to $X$,  endowed with
  the norm $$\|S\|_{L_{Q}}^{2}=tr(SQS^{\ast}) =\sum\limits_{k}| SQ^{\frac{1}{2}}\mathbf{e}_{k}|_{X}^{2}.$$
 Set $$L_2(H,X)=\left\{SQ^{\frac{1}{2}}: \, S\in L_{Q}(H_{0},X)\right\},$$ the norm is defined by
  $\|f\|^2_{L_2(H,X)}  =\sum\limits_{k}|f\mathbf{e}_{k}|_{X}^{2}$.

In this section, we assume that
$G\in L_2(H; H)$.
In section 4, we need the further assumption of $G\in L_2(H; H^\frac{5}{4})$. A typical example of $G$ could be taken the form
$$G=(I+A^\beta)^{-1},$$
for $\beta>0$.

\subsection{Well-posedness for stochastic equations} In this subsection, we will establish the basic qualitative properties: existence, uniqueness and continuity dependence of solution to equations \eqref{e1*}.  In the following, we focus on the necessary estimates.
\begin{lemma}\label{lem2.2*} Suppose that $\mathbf{u}^\varepsilon$ is the solution of equations \eqref{e1*}, then it satisfies for any $p\geq 1$
\begin{align}\label{2.10}
\mathrm{E}\sup_{t\in [0,T]}\|\mathbf{u}^\varepsilon\|_H^{2p}+\mathrm{E}\left(\int_{0}^{T}\|\mathbf{u}^\varepsilon\|_{H^\frac{5}{4}}^2dt\right)^p\leq C(T, p, \| G\|_{L_2(H,H)}, \|\mathbf{u}_0\|_{H}).
\end{align}
Particularly, for $p=1$, we have
\begin{align}\label{2.12}
\mathrm{E}\sup_{t\in [0,T]}\|\mathbf{u}^\varepsilon\|_H^{2}+\mathrm{E}\int_{0}^{T}\|\mathbf{u}^\varepsilon\|_{H^\frac{5}{4}}^2dt\leq TC\| G\|_{L_2(H,H)}+\|\mathbf{u}_0\|_{H}^2,
\end{align}
where $C$ is independence of $T, \varepsilon$, this means the bound is linear function with respect to $t$, which will be applied for the tightness argument of time average measure set.
\end{lemma}
\begin{proof}
Applying the It\^{o} formula to $\frac{1}{2}\|\mathbf{u}^\varepsilon\|_H^2$ and from \eqref{2.1}, we obtain
\begin{align}\label{a2.14}
\frac{1}{2}d\|\mathbf{u}^\varepsilon\|_H^2+\|\Lambda^\frac{5}{4}\mathbf{u}^\varepsilon\|_{L^2}^2dt=(\sqrt{\varepsilon}GdW, \mathbf{u}^\varepsilon)+\frac{\varepsilon}{2}\| G\|_{L_2(H,H)}^2 dt.
\end{align}
Using the BDG inequality,
\begin{align*}
&\mathrm{E}\sup_{t\in [0,T]}\left|\int_{0}^{t}(\sqrt{\varepsilon}GdW, \mathbf{u}^\varepsilon)\right|^p\nonumber\\
&\leq \mathrm{E}\left|\int_{0}^{T}\sum_{k\in \mathbb{Z}_0^3}(\sqrt{\varepsilon}Ge_k, \mathbf{u}^\varepsilon)^2dt\right|^\frac{p}{2}\nonumber\\
&\leq \frac{1}{4}\mathrm{E}\sup_{t\in [0,T]}\|\mathbf{u}^\varepsilon\|_H^{2p}+\varepsilon^\frac{p}{2}\mathrm{E}\left(\int_{0}^{T}\| G\|^2_{L_2(H,H)}dt\right)^\frac{p}{2}.
\end{align*}
Integrating of $t$, taking power $p$ and expectation of \eqref{a2.14}, we obtain \eqref{2.10}. When $p=1$, we could easily obtain \eqref{2.12}.
\end{proof}

\begin{lemma}\label{lem2.2} There exists a constant $C$ dependence of $T$ and initial data but independence of $\varepsilon$ such that
\begin{align*}
\mathrm{E}\int_{0}^{T}\|\mathbf{u}^\varepsilon\|_{H^{-\frac{5}{4}}}^2+\|d\mathbf{u}^\varepsilon/dt\|_{H^{-\frac{5}{4}}}^2dt\leq C.
\end{align*}
\end{lemma}
\begin{proof} Choosing $\theta\in H^\frac{5}{4}$, it holds
\begin{align}\label{2.12*}
 (\mathbf{u},\theta)=&(\mathbf{u}_0,\theta)-\int_0^t(A^\frac{5}{4} \mathbf{u}, \theta)dr-\int_{0}^{t}(B(\mathbf{u},\mathbf{u}), \theta)dr
 +\sqrt{\varepsilon}\int_0^t(GdW,\theta).
 \end{align}
 Using the H\"{o}lder inequality, we have
 \begin{align}
 &\mathrm{E}\int_{0}^{T}\int_0^t-(A^\frac{5}{4} \mathbf{u}, \theta)-(B(\mathbf{u},\mathbf{u}), \theta)drdt\nonumber\\
 &\leq \mathrm{E}\int_{0}^{T}\int_0^t(\|\mathbf{u}\|_{H^\frac{5}{4}}+\| \mathbf{u}\|_{H}\|\mathbf{u}\|_{H^\frac{5}{4}})\|\theta\|_{H^\frac{5}{4}}drdt\nonumber\\
 &\leq C(T)\|\theta\|_{H^\frac{5}{4}}\mathrm{E}\int_{0}^{T}1+\|\mathbf{u}\|_{H}^2+\| \mathbf{u}\|_{H^\frac{5}{4}}^2dt.
 \end{align}
By the martingale property, we see
\begin{align}\label{2.15*}
&\mathrm{E}\int_{0}^{T}\left|\int_0^t(GdW,\theta)dr\right|dt\nonumber\\
&\leq T\mathrm{E}\left(\int_{0}^{T}\|G\|^2_{L_2(H;H)}\|\theta\|_{L^2}^2dt\right)^\frac{1}{2}\nonumber\\
&\leq C(T)\|\theta\|_{L^2}\|G\|_{L_2(H;H)}.
\end{align}
Combining \eqref{2.12*}-\eqref{2.15*}, we obtain
$$\mathrm{E}\int_{0}^{T}\|\mathbf{u}^\varepsilon\|_{H^{-\frac{5}{4}}}^2dt\leq C.$$
By same argument, we also have
$$\mathrm{E}\int_{0}^{T}\|d\mathbf{u}^\varepsilon/dt\|_{H^{-\frac{5}{4}}}^2dt\leq C.$$
This completes the proof.
\end{proof}
\begin{proposition} For a stochastic basis $(\Omega, \mathcal{F}, \mathrm{P}, \{\mathcal{F}\}_{t\geq 0}, W)$, suppose that the initial data $\mathbf{u}_0\in H$ and $G\in L_2(H;H)$, then equations \eqref{e1*} admit a weak solution $\mathbf{u}$ which is $H$-valued progressively measurable process with regularity
$$\mathbf{u}\in L^p(\Omega; C([0,T];H)\cap L^2(0,T; H^\frac{5}{4})),$$
for any $p\geq 1$, and for any $\theta\in H^\frac{5}{4}$, it holds $\mathrm{P}$ a.s.
\begin{align}\label{2.9}
(\mathbf{u},\theta)=&(\mathbf{u}_0,\theta)-\int_0^t(A^\frac{5}{4} \mathbf{u}, \theta)dr-\int_{0}^{t}(B(\mathbf{u},\mathbf{u}), \theta)dr
 +\sqrt{\varepsilon}\int_0^t(GdW,\theta).
\end{align}
Furthermore, the solution is unique in the following sense: if $\mathbf{u}_1$ and $\mathbf{u}_2$ satisfy \eqref{2.9} with $\mathbf{u}_1(0)=\mathbf{u}_2(0)$, then
$$\mathrm{P}\{\mathbf{u}_1(t)=\mathbf{u}_2(t), ~for~ all~t\geq 0 \}=1.$$
\end{proposition}
We remark that the solution is weak in the sense of PDEs, while it is strong in the sense of probability, thus the solution is established on a fixed probability space.

\begin{proof} \underline{{\rm Existence}} The existence of proof follows from three steps: constructing Galerkin approximation solutions and the a priori estimates, stochastic compactness argument and passing the limit. Here, the a priori estimates were given above, the remaining steps is standard, we do not give the details, see \cite{Debussche, bre} for 2D case.

\underline{{\rm Uniqueness}} Denote by $\mathbf{u}$ the difference of two solutions $\mathbf{u}_1$ and $\mathbf{u}_2$, which satisfies
\begin{align*}
d\mathbf{u}+A^\frac{5}{4} \mathbf{u}dt+(B(\mathbf{u}_1,\mathbf{u}_1)-B(\mathbf{u}_2,\mathbf{u}_2))dt=0,
\end{align*}
with initial data $\mathbf{u}_0=0$.
Taking inner product with $\mathbf{u}$, we see
\begin{align}\label{2.11}
d\|\mathbf{u}\|_{H}^2+2\|\Lambda^\frac{5}{4} \mathbf{u}\|_{L^2}^2dt=-2(B(\mathbf{u}_1,\mathbf{u}_1)-B(\mathbf{u}_2,\mathbf{u}_2), \mathbf{u})dt.
\end{align}
Using \eqref{2.1}, the H\"{o}lder inequality and embedding $H^\frac{5}{4}\rightarrow L^{12}$, $H^\frac{5}{4}\rightarrow H^{1, \frac{12}{5}}$, we have
\begin{align}\label{2.120}
&(B(\mathbf{u}_1,\mathbf{u}_1)-B(\mathbf{u}_2,\mathbf{u}_2), \mathbf{u})=(B(\mathbf{u}, \mathbf{u}_1), \mathbf{u})\nonumber\\&\leq \|\Lambda\mathbf{u}\|_{L^\frac{12}{5}}\|\mathbf{u}_1\|_{L^{12}}\|\mathbf{u}\|_{H}\nonumber\\
&\leq \frac{1}{2}\|\Lambda^\frac{5}{4}\mathbf{u}\|_{L^2}^2
+\frac{1}{2}\|\mathbf{u}\|_{H}^2\|\mathbf{u}_1\|_{H^\frac{5}{4}}^2.
\end{align}
We have
\begin{align}\label{2.13}
d\|\mathbf{u}\|_{H}^2+\|\Lambda^\frac{5}{4} \mathbf{u}\|_{L^2}^2dt\leq\|\mathbf{u}\|_{H}^2\|\mathbf{u}_1\|_{H^\frac{5}{4}}^2dt.
\end{align}
Denote by
$$\rho(t)=\|\mathbf{u}_1(t)\|_{H^\frac{5}{4}}^2,$$
we could apply the It\^o product formula to function
$${\rm exp}\left(-\int_{0}^{t}\rho(r)dr\right)\|\mathbf{u}\|_{H}^2,$$
obtaining from \eqref{2.11}-\eqref{2.13}
\begin{align*}
&d{\rm exp}\left(-\int_{0}^{t}\rho(r)dr\right)\|\mathbf{u}\|_{H}^2\nonumber\\
&=-\rho(t){\rm exp}\left(-\int_{0}^{t}\rho(r)dr\right)\|\mathbf{u}\|_{H}^2+{\rm exp}\left(-\int_{0}^{t}\rho(r)dr\right)d\|\mathbf{u}\|_{H}^2\nonumber\\
&\leq 0.
\end{align*}
Then, integrating of $t$ and taking expectation yield
\begin{align*}
\mathrm{E}\left[{\rm exp}\left(-\int_{0}^{t}\rho(r)dr\right)\|\mathbf{u}\|_{H}^2\right]=0.
\end{align*}
Using the regularity $\mathbf{u}^1\in L^p(\Omega; L^2(0,T; H^\frac{5}{4}))$ for all $p\geq 2$, we know ${\rm exp}\left(-\int_{0}^{t}\rho(r)dr\right)>0$, $\mathrm{P}$ a.s. leading to
$$\mathrm{E}\|\mathbf{u}(t)\|_{H}^2=0.$$
We obtain the uniqueness.
\end{proof}

\subsection{Well-posedness for skeleton equations} In this subsection, we formulate the well-posedness for the skeleton equations
\begin{equation}\label{e1}
d\mathbf{u}+A^\frac{5}{4} \mathbf{u}dt+B(\mathbf{u},\mathbf{u})dt=G\varphi dt, ~\mathbf{u}(0)=\mathbf{u}_0,
\end{equation}
where $\varphi\in \mathcal{A}$ as the set of $H$-valued predictable stochastic process $\varphi$ such that $\int_{0}^{T}\|\varphi\|_{H}^{2}dt<\infty$, $\mathrm{P}$ a.s. For any fixed $M>0$, we define the set
\begin{eqnarray*}
S_{M}=\left\{\varphi\in L^{2}(0,T;H):\int_{0}^{T}\|\varphi\|_{H}^{2}dt\leq M\right\}.
\end{eqnarray*}
The set $S_{M}$ endows with the weak topology
$$
d(h,g)=\sum_{k\geq 1}\frac{1}{2^{k}}\left|\int_{0}^{T}\langle h(t)-g(t), \xi_{k}\rangle_{H}dt\right|,
$$
for $g, h\in S_M$, which is a Polish space and $\{\xi_{k}\}_{k\geq 1}$ is an orthonormal basis of $L^{2}(0,T;H)$. For $M>0$, define $\mathcal{A}_{M}=\{h\in \mathcal{A}:\varphi(\omega)\in S_{M}, {\rm a.s.}\}$.

Note that, here the skeleton equations is a deterministic Navier-Stokes equations with control term, the proof of well-posedness is easier compared with the stochastic version, therefore in the following we only establish several estimates used later.

\begin{lemma}\label{lem3.1} Suppose that $\mathbf{u}$ is the solution of skeleton equations \eqref{e1}, for any $\mathbf{u}_0\in H$, $\varphi\in \mathcal{A}_{M}$ and $G\in L_2(H; H)$, for any $T>0$ then
\begin{align}\label{3.2}
\sup_{t\in[0,T]}\|\mathbf{u}\|_{H}^2+\int_{0}^{T}\|\Lambda^\frac{5}{4} \mathbf{u}\|_{L^2}^2dt\leq \|\mathbf{u}_0\|_{H}^2+\int_{0}^{T}\|G\varphi\|_{H}^2dt,
\end{align}
and
\begin{align}\label{3.3}
\|\mathbf{u}(T)\|_{H}^2\leq C\left(\|\mathbf{u}_0\|_{H}^2+\int_{0}^{T}\frac{1}{\lambda_1}\|G\varphi\|_{H}^2dt\right){\rm exp}\left(-\frac{CT}{2}\right),
\end{align}
where $C$ is a constant and $\lambda_1$ is the first eigenvalue of $A$, actually in our case, $\lambda_1=1$.
We see from \eqref{3.3} that $\|\mathbf{u}(T)\|_{H}\rightarrow 0$ with the exponential decay speed when $T\rightarrow \infty$ complying with what is expected in the physical sense.

Furthermore, there exists constant $C$ such that the solution $\mathbf{u}$ has time regularity
\begin{align}\label{3.4*}\int_{0}^{T}\|\mathbf{u}\|^2_{H^{-\frac{5}{4}}}+\left\|d\mathbf{u}/dt\right\|^2_{H^{-\frac{5}{4}}}dt\leq C.\end{align}
\end{lemma}
\begin{proof} We only focus on the decay estimate \eqref{3.3}, while bounds \eqref{3.2} and \eqref{3.4*} could be obtained by same argument as \eqref{2.10} and Lemma \ref{lem2.2}. Taking inner with $\mathbf{u}$ in equations \eqref{e1}, we have
\begin{align*}
\frac{1}{2}d\|\mathbf{u}\|_H^2+\|\Lambda^\frac{5}{4}\mathbf{u}\|_{L^2}^2dt=(G\varphi, \mathbf{u})dt.
\end{align*}
Using the H\"{o}lder inequality and the Poincar\'{e} inequality, we have
\begin{align*}
(G\varphi, \mathbf{u})\leq \frac{1}{2}\|\mathbf{u}\|_{V}^2+\frac{C}{\lambda_1}\|G\varphi\|_{H}^2,
\end{align*}
as well as $\|\mathbf{u}\|_{L^2}^2\leq C\|\Lambda^\frac{5}{4}\mathbf{u}\|_{L^2}^2$ imply
\begin{align*}
\frac{1}{2}d\|\mathbf{u}\|_H^2&\leq -\frac{1}{2}\|\Lambda^\frac{5}{4}\mathbf{u}\|_{L^2}^2dt+\frac{C}{\lambda_1}\|G\varphi\|_{H}^2dt\nonumber\\
&\leq -\frac{1}{2C}\|\mathbf{u}\|_H^2dt+\frac{C}{\lambda_1}\|G\varphi\|_{H}^2dt.
\end{align*}
We infer from the Gronwall lemma
\begin{align*}
\|\mathbf{u}(T)\|_H^2&\leq {\rm exp}\left(-\int_{0}^{T}\frac{1}{2C}dt\right)\left(\|\mathbf{u}_0\|_{H}^2+\int_{0}^{T}\frac{C}{\lambda_1}\|G\varphi\|_{H}^2dt\right)\nonumber\\
&\leq C{\rm exp}\left(-\frac{T}{2C}\right)\left(\|\mathbf{u}_0\|_{H}^2+\int_{0}^{T}\frac{1}{\lambda_1}\|G\varphi\|_{H}^2dt\right).
\end{align*}
We finish the proof.
\end{proof}

\begin{proposition}Suppose that the initial data $\mathbf{u}_0\in H$ and $G\in L_2(H;H)$, then equations \eqref{e1} admit a unique weak solution $\mathbf{u}$  with regularity
$$\mathbf{u}\in C([0,T];H)\cap L^2(0,T; H^\frac{5}{4}),$$
 and for any $\theta\in H^\frac{5}{4}$, it holds
\begin{align*}
(\mathbf{u},\theta)=&(\mathbf{u}_0,\theta)-\int_0^t(A^\frac{5}{4} \mathbf{u}, \theta)dr-\int_{0}^{t}(B(\mathbf{u},\mathbf{u}), \theta)dr+\int_{0}^{t}(G\varphi, \theta)dr.
\end{align*}
Moreover, we have the solution is continuous of initial data, thus for $\mathbf{u}_{0,l}\rightarrow \mathbf{u}_{0}$ in $H$,
$$\mathbf{u}_l\rightarrow \mathbf{u} ~{\rm in}~ C([0,T];H).$$
\end{proposition}
\begin{proof} \underline{{\rm Existence and Uniqueness}} The proof of existence and uniqueness is easier than the stochastic case, here we do not give the details.

\underline{{\rm Continuity dependence}} The continuity dependence argument is similar to the uniqueness argument, we only give a simplify proof.  Let $\overline{\mathbf{u}}=\mathbf{u}_l-\mathbf{u}$, then
\begin{align*}
d\|\overline{\mathbf{u}}\|_{H}^2+2\|\Lambda^\frac{5}{4} \overline{\mathbf{u}}\|_{L^2}^2dt=-2(B(\mathbf{u}_l,\mathbf{u}_l)-B(\mathbf{u},\mathbf{u}), \overline{\mathbf{u}})dt.
\end{align*}
Since
\begin{align*}
2(B(\mathbf{u}_l,\mathbf{u}_l)-B(\mathbf{u},\mathbf{u}), \overline{\mathbf{u}})\leq \|\Lambda^\frac{5}{4} \overline{\mathbf{u}}\|_{L^2}^2+C\|\overline{\mathbf{u}}\|_{H}^2\|\mathbf{u}_1\|_{H^\frac{5}{4}}^2,
\end{align*}
 we have
\begin{align*}
\sup_{t\in [0,T]}\|\overline{\mathbf{u}}\|_{H}^2+2\int_{0}^{T}\|\Lambda^\frac{5}{4} \overline{\mathbf{u}}\|_{L^2}^2dt\leq\|\mathbf{u}_{0,l}- \mathbf{u}_{0}\|_{H}^2{\rm exp}\left(C\int_{0}^{T}\|\mathbf{u}_1\|_{H^\frac{5}{4}}^2dt\right).
\end{align*}
By the fact $\mathbf{u}_1\in  L^2(0,T; H^\frac{5}{4})$, the continuity follows.
\end{proof}

We also need the high-order regularity estimate with time-weight used for the proof of the large deviations upper bound.
\begin{lemma}\label{lem3.2*} There exists some certain constant $C$ such that the solution $\mathbf{u}$ of skeleton equations \eqref{e1} satisfies
\begin{align*}
&\sup_{t\in [0,T]}\|\sqrt{t}\mathbf{u}\|_{V}^2+\int_{0}^{T}\|\sqrt{t}\Lambda^\frac{9}{4}\mathbf{u}\|_{L^2}^2dt\nonumber\\
&\leq \left(C(T+1)\int^T_0\|G\varphi\|_{H}^2dt+\|\mathbf{u}_0\|^2_{H}\right){\rm exp}\left(C\|\mathbf{u}_0\|_{H}^2+C\int_{0}^{T}\|G\varphi\|_{H}^2dt\right).
\end{align*}
\end{lemma}
\begin{proof} Taking inner product with $\mathbf{u}$ in \eqref{e1}, using \eqref{2.1} we have
\begin{align*}
d\|\mathbf{u}\|_{H}^2+2\|\Lambda^\frac{5}{4}\mathbf{u}\|_{L^2}^2dt=2(G\varphi, \mathbf{u})dt.
\end{align*}
Taking integral of $t$, using the H\"{o}lder inequality and the Young inequality
\begin{align*}
\sup_{t\in [0,T]}\|\mathbf{u}\|_{H}^2+2\int_{0}^{T}\|\Lambda^\frac{5}{4}\mathbf{u}\|_{L^2}^2dt=\|\mathbf{u}_0\|_{H}^2+2\int_{0}^{T}(G\varphi, \mathbf{u})dt\nonumber\\
\leq \|\mathbf{u}_0\|_{H}^2+\frac{1}{2}\sup_{t\in [0,T]}\|\mathbf{u}\|_{H}^2+C\int_{0}^{T}\|G\varphi\|_{H}^2dt.
\end{align*}
Re-arranging the order, we have
\begin{align*}
\frac{1}{2}\sup_{t\in [0,T]}\|\mathbf{u}\|_{H}^2+2\int_{0}^{T}\|\Lambda^\frac{5}{4}\mathbf{u}\|_{L^2}^2dt\leq \|\mathbf{u}_0\|_{H}^2+C\int_{0}^{T}\|G\varphi\|_{H}^2dt.
\end{align*}

Taking differential to product function $t\|\Lambda \mathbf{u}\|_{L^2}^2$, we get
\begin{align}\label{3.8}
&d(t\|\Lambda \mathbf{u}\|_{L^2}^2)=td\|\Lambda \mathbf{u}\|_{L^2}^2+\|\Lambda \mathbf{u}\|_{L^2}^2dt\nonumber\\
&=-2t\|\Lambda^\frac{9}{4}\mathbf{u}\|_{L^2}^2dt-2t(B(\mathbf{u}, \mathbf{u}), A\mathbf{u})dt\nonumber\\
&\quad+2t(G\varphi, A\mathbf{u}) dt+\|\Lambda \mathbf{u}\|_{L^2}^2dt.
\end{align}
We proceed to estimate the nonlinear terms on the right hand side of \eqref{3.8}. Since

\begin{align}
\left|-2t(B(\mathbf{u}, \mathbf{u}), A\mathbf{u})\right|=2t\left|\int_{\mathbb{T}^3}\partial_k \mathbf{u}_i \partial_i \mathbf{u}_j \partial_k \mathbf{u}_jdx\right|\leq 2t\|\Lambda \mathbf{u}\|_{L^3}^3.
\end{align}
Then by Gagliardo–Nirenberg inequality
$$\|\Lambda \mathbf{u}\|_{L^3}\leq C\|\Lambda \mathbf{u}\|^\frac{1}{3}_{L^2}\|\Lambda^\frac{5}{4} \mathbf{u}\|^\frac{1}{3}_{L^2}\|\Lambda^\frac{9}{4}\mathbf{u}\|^\frac{1}{3}_{L^2},$$
we have
\begin{align}
\left|-2t(B(\mathbf{u}, \mathbf{u}), A\mathbf{u})\right|&\leq 2Ct\|\Lambda \mathbf{u}\|_{L^2}\|\Lambda^\frac{5}{4} \mathbf{u}\|_{L^2}\|\Lambda^\frac{9}{4}\mathbf{u}\|_{L^2}\nonumber\\
&\leq t\|\Lambda^\frac{9}{4}\mathbf{u}\|_{L^2}^2+Ct\|\Lambda \mathbf{u}\|^2_{L^2}\|\Lambda^\frac{5}{4} \mathbf{u}\|^2_{L^2}.
\end{align}

By the H\"{o}lder inequality and the Young inequality,
\begin{align}\label{3.11}
2t(G\varphi, A\mathbf{u}) \leq \frac{t}{2}\|A\mathbf{u}\|_{H}^2+Ct\|G\varphi\|_{H}^2.
\end{align}
Combining \eqref{3.8}-\eqref{3.11}, we have
\begin{align}\label{2.24}
&\left\|\sqrt{t}\|\Lambda \mathbf{u}\|_{L^2}\right\|_{L^\infty}^2+\int_{0}^{T}\frac{t}{2}\|\Lambda^\frac{9}{4}\mathbf{u}\|_{L^2}^2dt\nonumber\\
&\leq C\int^T_0t\|\Lambda\mathbf{u}\|^2_{L^2}\|\Lambda^\frac{5}{4} \mathbf{u}\|^2_{L^2}dt\nonumber\\
&\quad+C\int^T_0t\|G\varphi\|_{H}^2dt+\int_{0}^{T}\|\Lambda \mathbf{u}\|_{L^2}^2dt\nonumber\\
&\leq  C\int^T_0t\|\Lambda\mathbf{u}\|^2_{L^2}\|\Lambda^\frac{5}{4} \mathbf{u}\|^2_{L^2}dt+C\int^T_0t\|G\varphi\|_{H}^2dt+\|\mathbf{u}_0\|_{H}^2+C\int_{0}^{T}\|G\varphi\|_{H}^2dt.
\end{align}
By \eqref{3.2}, using the Gronwall lemma to \eqref{2.24}, we see
\begin{align*}
&\left\|\sqrt{t}\|\Lambda \mathbf{u}\|_{L^2}\right\|_{L^\infty}^2+\int_{0}^{T}\frac{t}{2}\|\Lambda^\frac{9}{4}\mathbf{u}\|_{L^2}^2dt\nonumber\\
&\leq  \left(C\int^T_0t\|G\varphi\|_{H}^2dt+\|\mathbf{u}_0\|_{H}^2+C\int_{0}^{T}\|G\varphi\|_{H}^2dt\right){\rm exp}\left(\int^T_0\|\Lambda^\frac{5}{4} \mathbf{u}\|^2_{L^2}dt\right)\nonumber\\
&\leq \left(C\int^T_0t\|G\varphi\|_{H}^2dt+\|\mathbf{u}_0\|_{H}^2+C\int_{0}^{T}\|G\varphi\|_{H}^2dt\right){\rm exp}\left(C\|\mathbf{u}_0\|_{H}^2+C\int_{0}^{T}\|G\varphi\|_{H}^2dt\right).
\end{align*}
This completes the proof.
\end{proof}

\section{The unique ergodicity}
In this section, our main goal is to establish the unique ergodicity of Markov semigroup of solution to equations \eqref{Equ1.1}. We first give a preliminary result concerning the existence of invariant measure in the first part. Then, in the second part, we will devote to prove the unique ergodicity of invariant measure by establishing the exponential stability result using the small noise perturbation.

\subsection{The existence of invariant measure}
In this subsection, we show the existence of invariant measure using the Maslowski-Seidler theory \cite{Mas}, which tell us a $bw$-Feller semigroup has an invariant probability measure provided the set
\begin{align}\label{2.101}
\left\{\frac{1}{T_n}\int_{0}^{T_n}\mathbf{P}_t^*\mu dt, n\geq 1\right\}
\end{align}
is tight on $(H, bw)$. Let us introduce the meaning of notations in \eqref{2.101}. Define by $\mathbf{P}_t(x, \cdot)$ the transition probability
$$\mathbf{P}_t(x, \mathcal{O})=\mathbf{P}(\mathbf{u}(t, x)\in \mathcal{O}),$$
for set $\mathcal{O}\in \mathcal{B}(H)$, where $\mathbf{u}(t, x)$ is the pathwise solution of system \eqref{Equ1.1} starting from the initial data $x$.

For any bounded Borel function $\Phi\in \mathcal{B}_b(H)$, define a Markov transition semigroup
\begin{align*}
(\mathbf{P}_t\Phi)(x)=\mathrm{E}[\Phi(\mathbf{u}(t, x))],~ x\in H.
\end{align*}
Denote by $\mathbf{P}_t^*$ the dual of transition semigroup $\mathbf{P}_t$.
We say a probability measure $\mu$ on $\mathcal{B}(H)$ is an invariant measure if
$$\int_{H}\mathbf{P}_t\Phi d\mu=\int_{H}\Phi d\mu,~~ {\rm for~ all}~t\geq 0, ~\Phi\in \mathcal{B}_b(H).$$
An invariant measure $\mu$ is ergodic, if for all $\Phi\in L^2(H,\mu)$, we have
\begin{align*}
\lim_{T\rightarrow\infty}\frac{1}{T}\int_{0}^{T}\mathbf{P}_t\Phi dt=\int_{H}\Phi(\mathbf{u}) d\mu(\mathbf{u}),~ {\rm in}~L^2(H,\mu).
\end{align*}

We give more details of the Maslowski-Seidler theory for establishing the existence of invariant measure.
\begin{proposition}\label{pro5.1} \cite[Proposition 3.1]{Mas} Suppose that the semigroup $\mathbf{P}_t$ is sequentially weakly Feller, that is, $$\mathbf{P}_t: C_b(H_w)\rightarrow \mathcal{L}_b(H_w).$$
And assume that we can find a Borel probability measure $\nu$ on $H$ and $T_0>0$ such that for any $\varepsilon>0$ there exists $R>0$ satisfying
$$\sup_{T>T_0}\frac{1}{T}\int_{0}^{T}(\mathbf{P}_t^*\nu)(S)dt\leq \varepsilon,$$
where the set $S:=\left\{\mathbf{u}: \|\mathbf{u}\|_{H}>R\right\}$ for a certain constant $R>0$ and $\mathbf{P}_t^*$
is the dual of semigroup of $\mathbf{P}_t$.
Then, there exists an invariant measure for the semigroup $\mathbf{P}_t$.
\end{proposition}
\begin{remark} Note that, the Maslowski-Seidler theory extended the classical Krylov-Bogoliubov theory, which relaxes the Feller condition to sequentially weak Feller. As a result, we do not need the higher-order energy estimates which simplifies the proof.
\end{remark}

\begin{proposition}\label{pro3.1**} Suppose that $G\in L_2(H; H)$ holds. Then, the transition semigroup $\mathbf{P}_t$ has an invariant measure $\mu$.
\end{proposition}
\begin{proof} We show that the semigroup $\mathbf{P}_t$ is sequentially weak Feller. Corresponding to the sequence of initial data $\mathbf{u}_{0,l}\in H$, there exists a sequence $H$-valued $\mathcal{F}_t$-progressive measurable processes $\mathbf{u}_{l}$ as the solutions of equations \eqref{Equ1.1}. As Lemma \ref{lem2.2*} and Lemma \ref{lem2.2},  the family $\{\mathbf{u}_{l}\}_{l\geq 1}$ has uniform bound in $L^p(\Omega; C([0,T]; H)\cap L^2(0,T; H^\frac{5}{4}))$ and $L^2(\Omega; W^{1,2}(0,T; H^{-\frac{5}{4}})$ with respect to $l$. Using the Aubin-Lions lemma, we could infer that the law of the family $\{\mathbf{u}_{l}\}_{l\geq 1}$ is tight in $\mathcal{X}$ where $\mathcal{X}=C([0,T]; H_w)\cap L^2(0,T; H)$.

 Furthermore, the Skorokhod representation theorem implies that  there exist a new probability space
$(\widetilde{\Omega}, \widetilde{\mathcal{F}}, \widetilde{\mathrm{P}})$, a new subsequence $\widetilde{\mathbf{u}}_{l_k}$ and the process $\widetilde{\mathbf{u}}$ such that
$$\widetilde{\mathbf{u}}_{l_k}~{\rm and} ~\mathbf{u}_{l_k}, ~ \widetilde{\mathbf{u}}~{\rm and}~\mathbf{u} ~{\rm have ~ the ~same ~joint~ distribution~ in} ~\mathcal{X},$$
and
$$\widetilde{\mathbf{u}}_{l_k}\rightarrow \widetilde{\mathbf{u}} ~{\rm in~ the~ topology ~of~ }\mathcal{X}, ~\widetilde{\mathrm{P}}\mbox{ a.s.}$$

The convergence  together with the fact that $\phi$ is a bounded sequentially weakly continuous function yields
$$\phi(\widetilde{\mathbf{u}}_{l_k})\rightarrow \phi(\widetilde{\mathbf{u}})~ {\rm in} ~\mathbb{R}, ~\widetilde{\mathrm{P}}\mbox{ a.s.}$$
The fact that the processes $\widetilde{\mathbf{u}}_{l_k}$ and $\mathbf{u}_{l_k}$, $\widetilde{\mathbf{u}}$ and $\mathbf{u}$ having the same distribution leads to
\begin{align}
&\widetilde{\mathrm{E}}\left[\phi(\widetilde{\mathbf{u}}_{l_k}(t;\widetilde{\mathbf{u}}_{0,l_k}))\right]
=\mathrm{E}\left[\phi(\mathbf{u}_{l_k}(t;\mathbf{u}_{0,l_k}))\right]
=(\mathbf{P}_t\phi)(\mathbf{u}_{0,l_k}),\label{5.2*}\\
&\widetilde{\mathrm{E}}\left[\phi(\widetilde{\mathbf{u}}(t;\widetilde{\mathbf{u}}_0))\right]
=\mathrm{E}\left[\phi(\mathbf{u}(t;\mathbf{u}_0))\right]=(\mathbf{P}_t\phi)(\mathbf{u}_{0}).\label{5.3*}
\end{align}
Therefore, by \eqref{5.2*} and \eqref{5.3*}, we have
$$\lim_{k\rightarrow\infty}(\mathbf{P}_t\phi)(\mathbf{u}_{0,l_k})=(\mathbf{P}_t\phi)(\mathbf{u}_{0}).$$
Using the sub-subsequence argument, we obtain that the original sequence satisfies
$$\lim_{l\rightarrow\infty}(\mathbf{P}_t\phi)(\mathbf{u}_{0,l})=(\mathbf{P}_t\phi)(\mathbf{u}_{0}).$$
Obviously, $\mathbf{P}_t\phi$ from $H$ into $\mathbb{R}$ is bounded. We conclude that the semigroup $\mathbf{P}_t$ is sequentially weak Feller, thus,
 $$\mathbf{P}_t: C_b(H_w)\rightarrow C_b(H_w).$$

 Then, it enough to show the time average measure set \eqref{2.101} is tight on $(H, bw)$. By the Chebyshev inequality, the Poincar\'{e} inequality  and Lemma \ref{lem2.2*}, we get
\begin{align*}
\frac{1}{T}\int_{0}^{T}(\mathbf{P}_t^*\delta_{x})(H\setminus B_R)dt&=\frac{1}{T}\int_{0}^{T}\mathrm{P}\{\|\mathbf{u}\|_{H}>R\}dt\nonumber\\
&\leq \frac{1}{R^2T}\int_{0}^{T}\mathrm{E}\|\mathbf{u}\|_{H}^2dt\nonumber\\
&\leq \frac{C}{R^2T}\int_{0}^{T}\mathrm{E}\|\Lambda^\frac{5}{4} \mathbf{u}\|_{L^2}^2dt\nonumber\\
&\leq \frac{C+CT}{R^2T},
\end{align*}
where the set $B_R:=\{\mathbf{u}:\|\mathbf{u}\|_{H}\leq R\}$. Then, the existence of invariant measure follows from the Maslowski and Seidler theory, Proposition \ref{pro5.1}. (see also Proposition 3.1 in \cite{Mas}).
\end{proof}

\subsection{The unique ergodicity} The uniqueness argument of invariant measure is a much more challenging topic. For the non-degenerate noise, the unique ergodicity could be achieved generally by two classical methods: the first method is to establish the exponential stability result; Second one should be more probabilistic
arguments, that is, to prove strong Feller property and irreducible of transition semigroup $\mathbf{P}_t$, see \cite[Section 7]{dz}. However, the strong Feller property fails to hold when the noise is spatially degeneration, Hairer and  Mattingly introduced the concept of asymptotic strong Feller to cope with this problem in \cite{Hai-Mat}. Alternatively, for moderately degenerate noise, an asymptotic coupling method that has shown effectively for the proof of ergodicity was developed by \cite{Ha,ma}.

 Here, the unique ergodcity as an auxiliary result, we give the straightforward deterministic argument which relies on the following exponential moment and exponential stability estimates.
 \begin{lemma}\label{lem2.1*} There exists a constant $C$ independence of $T$ such that the solution $\mathbf{u}^\varepsilon$ of equations \eqref{e1*} satisfies exponential moment
 \begin{align*}
 \mathrm{E}{\rm exp}\left(\|\mathbf{u}^\varepsilon\|_{H}^2+\int_{0}^{T}\|\Lambda^\frac{5}{4}\mathbf{u}^\varepsilon\|_{L^2}^2dt\right)\leq {\rm exp}\left(\|\mathbf{u}^\varepsilon_0\|_{H}^2\right)+{\rm exp}\left(C\varepsilon T\|G\|^2_{L_2(H; H)} \right).
 \end{align*}
 \end{lemma}
\begin{proof} Let
$$\Psi(t):=\|\mathbf{u}^\varepsilon(t)\|_{H}^2+\int_{0}^{t}\|\Lambda^\frac{5}{4}\mathbf{u}^\varepsilon\|_{L^2}^2ds.$$
Using the It\^{o} formula to ${\rm exp}^{\Psi(t)}$ and \eqref{a2.14}, we have
\begin{align*}
{\rm exp}^{\Psi(t)}=&{\rm exp}^{\Psi(0)}-\int_{0}^{t}{\rm exp}^{\Psi(s)}\|\Lambda^\frac{5}{4}\mathbf{u}^\varepsilon\|_{L^2}^2ds\nonumber\\
&+2\int_{0}^{t}{\rm exp}^{\Psi(s)}(\sqrt{\varepsilon}GdW, \mathbf{u}^\varepsilon)+\int_{0}^{t}\varepsilon{\rm exp}^{\Psi(s)}\|G\|^2_{L_2(H; H)} ds\nonumber\\
&+2\int_{0}^{t}{\rm exp}^{\Psi(s)}\sum_{k\in \mathbb{Z}_0^3}(\sqrt{\varepsilon}Ge_k, \mathbf{u}^\varepsilon)^2ds.
\end{align*}
The H\"{o}lder inequality yields
\begin{align}\label{3.5}
2\int_{0}^{t}{\rm exp}^{\Psi(s)}\sum_{k\in \mathbb{Z}_0^3}(\sqrt{\varepsilon}Ge_k, \mathbf{u}^\varepsilon)^2ds
\leq 2\varepsilon\int_{0}^{t}{\rm exp}^{\Psi(s)}\|\mathbf{u}^\varepsilon\|_{H}^2\|G\|^2_{L_2(H; H)}ds.
\end{align}
Using the Poincar\'{e} inequality and \eqref{3.5}, we obtain
\begin{align*}
&-\int_{0}^{t}{\rm exp}^{\Psi(s)}\|\Lambda^\frac{5}{4}\mathbf{u}^\varepsilon\|_{H}^2ds+2\int_{0}^{t}{\rm exp}^{\Psi(s)}\sum_{k\in \mathbb{Z}_0^3}(\sqrt{\varepsilon}Ge_k, \mathbf{u}^\varepsilon)^2ds\nonumber\\
&\leq \int_{0}^{t}{\rm exp}^{\Psi(s)}\left(-\|\Lambda^\frac{5}{4}\mathbf{u}^\varepsilon\|_{L^2}^2+2\varepsilon\|\mathbf{u}^\varepsilon\|_{H}^2\|G\|^2_{L_2(H; H)} \right)ds\nonumber\\
&\leq \int_{0}^{t}{\rm exp}^{\Psi(s)}\left(-1+2\varepsilon\|G\|^2_{L_2(H; H)} \right)\|\mathbf{u}^\varepsilon\|^2_{H}ds.
\end{align*}
Since we study the asymptotic behaviour of solutions as $\varepsilon\rightarrow 0$, we could choose $\varepsilon$ small enough such that
$$-1+2\varepsilon\|G\|^2_{L_2(H; H)}<0,$$
which implies
\begin{align}\label{3.5*}
{\rm exp}^{\Psi(t)}\leq {\rm exp}^{\Psi(0)}+2\int_{0}^{t}{\rm exp}^{\Psi(s)}(\sqrt{\varepsilon}GdW, \mathbf{u}^\varepsilon)+\int_{0}^{t}\varepsilon{\rm exp}^{\Psi(s)}\|G\|^2_{L_2(H; H)}ds.
\end{align}

For any fixed $N>0$, define by $\tau_N$ the stopping time
$$\tau_{N}=\inf\left\{t>0, \|\mathbf{u}^\varepsilon\|_{H}+\int_{0}^{t}\|\Lambda^\frac{5}{4}\mathbf{u}^\varepsilon\|_{L^2}^2ds>N\right\},$$
on $[0, \tau_{N}\wedge t]$, we have from \eqref{3.5*}
\begin{align}\label{3.6}
\mathrm{E}{\rm exp}^{\Psi(\tau_{N}\wedge t)}\leq {\rm exp}^{\Psi(0)}+\mathrm{E}\int_{0}^{\tau_{N}\wedge t}\varepsilon{\rm exp}^{\Psi(s)}\|G\|^2_{L_2(H; H)}  ds.
\end{align}
Finally, we have by the Gronwall lemma and passing $N\rightarrow\infty$ in \eqref{3.6}
\begin{align*}
\mathrm{E}{\rm exp}^{\Psi(t)}\leq {\rm exp}^{\Psi(0)}+{\rm exp}^{C\varepsilon t\|G\|^2_{L_2(H; H)} },
\end{align*}
where $C$ is independence of $t$. We complete the proof.
\end{proof}

\begin{lemma}\label{lem3.2}  Assume that $\mathbf{u}^\varepsilon_1$ and $\mathbf{u}^\varepsilon_2$ are two solutions of equations \eqref{e1*} corresponding to the initial data $\mathbf{u}_1(0)$ and $\mathbf{u}_2(0)$, then the exponential stability holds
\begin{align*}
\mathrm{E}\|\mathbf{u}^\varepsilon_1-\mathbf{u}^\varepsilon_2\|_{H}^2\leq \|\mathbf{u}_1(0)-\mathbf{u}_2(0)\|_{H}^2{\rm exp}\left(-kt\right),
\end{align*}
for $k$ being positive constant.
\end{lemma}

 \begin{proof}
 The proof is almost same with the continuous dependence argument. We first have
\begin{align}\label{3.7*}
&\frac{1}{2}d\|\mathbf{u}^\varepsilon_1-\mathbf{u}^\varepsilon_2\|_{H}^2+\|\Lambda^\frac{5}{4}(\mathbf{u}^\varepsilon_1-\mathbf{u}^\varepsilon_2)\|_{L^2}^2dt\leq -(B(\mathbf{u}_1-\mathbf{u}_2, \mathbf{u}_1), \mathbf{u}_1-\mathbf{u}_2)dt.
\end{align}
As \eqref{2.120}, we have
\begin{align}\label{3.8*}
(B(\mathbf{u}_1-\mathbf{u}_2, \mathbf{u}_1), \mathbf{u}_1-\mathbf{u}_2)&\leq \|\nabla(\mathbf{u}_1-\mathbf{u}_2)\|_{L^\frac{12}{5}}\|\mathbf{u}_1\|_{L^{12}}\|\mathbf{u}_1-\mathbf{u}_2\|_{H}\nonumber\\
&\leq \frac{1}{2}\|\Lambda^\frac{5}{4}(\mathbf{u}^\varepsilon_1-\mathbf{u}^\varepsilon_2)\|_{L^2}^2
+\frac{1}{2}\|\mathbf{u}_1-\mathbf{u}_2\|_{H}^2\|\mathbf{u}_1\|_{H^\frac{5}{4}}^2,
\end{align}
then, using \eqref{3.7*}, \eqref{3.8*}, we have
\begin{align*}
&\frac{1}{2}d\|\mathbf{u}^\varepsilon_1-\mathbf{u}^\varepsilon_2\|_{H}^2\leq -\frac{1}{2}\|\Lambda^\frac{5}{4}(\mathbf{u}^\varepsilon_1-\mathbf{u}^\varepsilon_2)\|_{L^2}^2dt +\frac{1}{2}\|\mathbf{u}_1-\mathbf{u}_2\|_{H}^2\|\mathbf{u}_1\|_{H^\frac{5}{4}}^2dt.
\end{align*}
The Poincar\'{e} inequality and the Gronwall lemma, Lemma \ref{lem2.1*} yield
\begin{align*}
\mathrm{E}\|\mathbf{u}^\varepsilon_1-\mathbf{u}^\varepsilon_2\|_{H}^2&\leq C\|\mathbf{u}_1(0)-\mathbf{u}_2(0)\|_{H}^2\mathrm{E}{\rm exp}\left(\int_{0}^{t} -\frac{1}{2}+\frac{1}{2}\|\mathbf{u}_1\|_{H^\frac{5}{4}}^2ds\right)\nonumber\\
&\leq C\|\mathbf{u}_1(0)-\mathbf{u}_2(0)\|_{H}^2{\rm exp}\left(-t\right)\left({\rm exp}\left(\|\mathbf{u}^\varepsilon_0\|_{H}^2\right)+{\rm exp}\left(C\varepsilon t\|G\|^2_{L_2(H; H)}\right)\right)\nonumber\\
&\leq C\|\mathbf{u}_1(0)-\mathbf{u}_2(0)\|_{H}^2{\rm exp}\left(-t\right)\nonumber\\&\quad +C\|\mathbf{u}_1(0)-\mathbf{u}_2(0)\|_{H}^2{\rm exp}\left(-t+C\varepsilon t\|G\|^2_{L_2(H; H)}\right).
\end{align*}
Let
$$k=1-C\varepsilon\|G\|^2_{L_2(H; H)},$$
as $\varepsilon$ being small, then $k$ is positive. We finish the proof.
 \end{proof}
With the exponential stability in hands, we could easily show the invariant measure is unique. Assume that $\nu$ is another invariant measure, applying the invariance and Lemma \ref{lem3.2}, we see
\begin{align}\label{2.46}
&\left|\int_{H}\phi(x) d\mu(x)-\int_{H}\phi(y) d\nu(y)\right|\nonumber\\ &=\left|\int_{H}\mathbf{P}_t\phi(x) d\mu(x)-\int_{H}\mathbf{P}_t\phi(y) d\nu(y)\right|\nonumber\\
&=\left|\int_{H}\int_{H}\mathbf{P}_t\phi(x) -\mathbf{P}_t\phi(y)  d\mu(x)d\nu(y)\right|\nonumber\\
&=\left|\int_{H}\int_{H}\mathrm{E}\phi(\mathbf{u}(t;x)) -\mathrm{E}\phi(\mathbf{u}(t;y)) d\mu(x)d\nu(y)\right|\nonumber\\
&\leq \|\phi\|_{Lip}\left|\int_{H}\int_{H}\mathrm{E}\|\mathbf{u}(t;x)-\mathbf{u}(t;y)\|_{H}d\mu(x)d\nu(y)\right|\nonumber\\
&\leq C\|\phi\|_{Lip}{\rm exp}\left(-kt\right)\left|\int_{H}\int_{H}\|x-y\|_{H}^2d\mu(x)d\nu(y)\right|.
\end{align}
As $t\rightarrow \infty$, the right hand side term of \eqref{2.46} goes to zero, which implies the invariant measure $ \mu$ is unique. Following \cite[Theorem 3.2.6]{Zabczyk}, it is ergodic.

\section{Large deviations of linear stochastic equations}

In this section, we establish the large deviations of the distribution of solutions to the linear stochastic equations
\begin{align}\label{e2.1}
d\mathbf{v}+A^\frac{5}{4} \mathbf{v}dt=\sqrt{\varepsilon}GdW,~~ \mathbf{v}(0)=0.
\end{align}
Under the condition of $G\in L_2(H; H^\frac{5}{4})$, we could deduce that equations \eqref{e2.1} admit a unqiue global strong pathwise solution $\mathbf{v}^\varepsilon\in L^p(\Omega; C([0,T]; H^\frac{5}{4})\cap L^2(0,T; H^\frac{5}{2}))$ uniformly in $\varepsilon$ for any $p\geq 2$. Since the equations are linear, we do not give more details on the well-posedness argument. We proceed to show that the distribution of the family of $\{\mathbf{v}^\varepsilon\}_{\varepsilon>0}$ satisfies the large deviations in space $C([0,T]; H^\frac{5}{4})\cap L^2(0,T; H^\frac{5}{2})$.

Considering the skeleton equations
\begin{align}\label{e2.2}
d\mathbf{v}+A^\frac{5}{4} \mathbf{v}dt=G\varphi dt,~~ \mathbf{v}(0)=0.
\end{align}
Note that for any $G\in L_2(H; H^\frac{5}{4})$, we could deduce that equations \eqref{e2.2} also admit a unique global strong solution $\mathbf{v}\in C([0,T]; H^\frac{5}{4})\cap L^2(0,T; H^\frac{5}{2})$. Actually, we could establish higher-order regularity estimate for $\mathbf{v}$: \begin{align}\label{4.3*}\mathbf{v}\in C([0,T]; H^\frac{5}{2})\cap L^2(0,T; H^\frac{15}{4}).\end{align}
Indeed, applying $\Lambda^\frac{5}{4}$ on both sides of \eqref{e2.2}, taking inner with $\Lambda^\frac{15}{4}\mathbf{v}$ and integrating by parts, we have
\begin{align*}
d\|\Lambda^\frac{5}{2}\mathbf{v}\|_{L^2}^2+2\|\Lambda^\frac{15}{4}\mathbf{v}\|_{L^2}^2&=2(\Lambda^\frac{5}{4} G\varphi, \Lambda^\frac{15}{4}\mathbf{v}) dt\nonumber\\
&\leq \|\Lambda^\frac{15}{4}\mathbf{v}\|_{L^2}^2dt+C\|G\|_{L_2(H; H^\frac{5}{4})}^2\|\varphi\|_{H}^2dt,
\end{align*}
since $\varphi\in L^2(0,T;H)$, the estimate follows after taking integral of $t$.

Next, we formulate the preliminaries of large deviations. For a Polish space $\mathcal{X}$, a function $I:\mathcal{X}\rightarrow [0,\infty]$ is called a rate function if $I$ is lower semicontinuous and is referred to as a good rate function if for each $M<\infty$, the level set $\{x\in \mathcal{X}:I(x)\leq M\}$ is compact. For completeness we now give the definition of large deviations and Laplace principles. For more backgrounds in this area of study we refer to \cite{Ellis}.

\begin{definition}[Large Deviations] The family $\left\{X^{\varepsilon} \right\}_{\varepsilon>0}$ satisfies the large deviations on $\mathcal{X}$ with rate function $I$ if the following two conditions hold:

i. lower bound: for every open set $\mathcal{O}\subset \mathcal{X}$,
\begin{equation*}
-\inf_{x\in \mathcal{O}} I(x) \leq \liminf_{\varepsilon \rightarrow 0}\varepsilon \log \mathrm{P}(X^{\varepsilon} \in \mathcal{O});
\end{equation*}

ii. upper bound: for every closed set $\mathcal{C} \subset \mathcal{X}$,
\begin{equation*}
\limsup_{\varepsilon \rightarrow 0} \varepsilon \log \mathrm{P}(X^{\varepsilon} \in \mathcal{C}) \leq -\inf_{x\in\mathcal{ C}}I(x).
\end{equation*}
\end{definition}

\begin{definition}[Laplace Principle]\label{def4.2} Let $I$ be a rate function on space $\mathcal{X}$. A family $\{X^{\varepsilon}\}_{\varepsilon>0}$ of $\mathcal{X}$-valued random processes is said to satisfy the Laplace principle on $\mathcal{X}$ with a rate function $I$ if for each real-valued, bounded and continuous function $f$, we have
\begin{eqnarray*}
\lim_{\varepsilon\rightarrow 0}\varepsilon\log \mathrm E\bigg\{{\rm exp}\bigg[-\frac{1}{\varepsilon}f(X^{\varepsilon})\bigg]\bigg\}=-\inf_{x\in \mathcal{X}}\{f(x)+I(x)\}.
\end{eqnarray*}
\end{definition}
Since the family $\{X^{\varepsilon}\}_{\varepsilon>0}$ is a Polish space valued random process, the Laplace principle and the large deviation principle are equivalent, see \cite[Theorem 1.2.3]{Ellis}. To apply the weak convergence approach, we will use the following theorem given in \cite{Dup} to show the Laplace principle, then the large deviations follows.
\begin{theorem}\cite[Theorem 6]{Dup} \label{the4.1} For Polish spaces $\mathcal{X},\mathcal{Y}$ and each $\varepsilon>0$, let $\mathcal{G}^{\varepsilon}:\mathcal{Y}\rightarrow \mathcal{X}$ be the solution mapping acting on the noise
for fixed initial conditions and define $X^{\varepsilon}:=\mathcal{G}^{\varepsilon}(\sqrt{\varepsilon}W)$ where $W$ is a Wiener process. If there is a measurable map $\mathcal{G}^{0}:\mathcal{Y}\rightarrow \mathcal{X}$ such that the following conditions hold:

$(1)$ For $M<\infty$, if $h_{\varepsilon}$ converges in distribution to $\varphi$ as $S_{M}$-valued random elements, then,
\begin{eqnarray*}
\mathcal{G}^{\varepsilon}\left(\sqrt{\varepsilon}W+\int_{0}^{\cdot}\varphi_{\varepsilon}(s)ds\right)\rightarrow \mathcal{G}^{0}\left(\int_{0}^{\cdot}\varphi ds\right)
\end{eqnarray*}
as $\varepsilon\rightarrow 0$ in distribution $\mathcal{X}$;

$(2)$ For every $M<\infty$, the set
\begin{eqnarray*}
K_{M}=\{X_{\varphi}:\varphi\in S_{M}\}
\end{eqnarray*}
is a compact subset of $\mathcal{X}$.
Then, the family $\{X^{\varepsilon}\}_{\varepsilon>0}$ satisfies the Laplace principle with the rate function
\begin{eqnarray*}
I(X)=\inf_{\left\{\varphi\in L^{2}(0,T;H):X=\mathcal{G}^{0}(\int_{0}^{\cdot}\varphi(s)ds)\right\}}\left\{\frac{1}{2}\int_{0}^{T}\|\varphi\|_{H}^{2}dt\right\}.
\end{eqnarray*}
\end{theorem}
In our setting, we choose the Polish spaces $\mathcal{Y}, \mathcal{X}$ to be $C([0,T]; H), C([0,T]; H^\frac{5}{4})\cap L^2(0,T; H^\frac{5}{2})$. In the followings, we will verify the conditions in Theorem \ref{the4.1}. First, denote by $\mathbf{v}=\mathcal{G}^{\varepsilon}(\sqrt{\varepsilon}W)$ be the solution of system \eqref{e2.1}. Let $\{\varphi_{\varepsilon}\}_{\varepsilon\in(0,1]}\subset \mathcal{A}_{M}$ be a family of random elements and denote by $\mathbf{v}^\varepsilon$ to be the solution of the following stochastic controlled equations
\begin{equation}\label{e3}
d\mathbf{v}^\varepsilon+A^\frac{5}{4} \mathbf{v}^\varepsilon dt=G\varphi_\varepsilon dt+\sqrt{\varepsilon}GdW,~~ \mathbf{v}^\varepsilon(0)=0.
\end{equation}
Owing to the uniqueness, we know $\mathbf{v}^\varepsilon=\mathcal{G}^{\varepsilon}\left(\sqrt{\varepsilon}W+\int_{0}^{\cdot}\varphi_{\varepsilon}(s)ds\right)$.

In order to obtain the compactness, we introduce the operator $\Gamma: L^2(0,T; H)\rightarrow C([0,T]; H^\frac{5}{4})$ by
\begin{align}\label{2.3*}\Gamma \varphi=\int_{0}^{t}G\varphi(s) ds,\end{align}
for any $\varphi\in L^2(0,T; H)$.

\begin{lemma}\label{lem2.1}  The operator $\Gamma$ is compact with respect to the topology of $C([0,T]; H^\frac{5}{4})$, that is, for any bound sequence $\varphi_n\in L^2(0,T; H)$ with $\varphi_n\rightharpoonup \varphi$ in $L^2(0,T; H)$, it holds $\Gamma \varphi_n\rightarrow\Gamma \varphi$ in $C([0,T]; H^\frac{5}{4})$.
\end{lemma}
\begin{proof} Since the operator $G\in L_2(0,T; H^\frac{5}{4})$, we have the $\Gamma$ is a bounded linear operator from $L^2(0,T; H)$ into $C([0,T]; H^\frac{5}{4})$. Therefore, we could have $\Gamma \varphi_n\rightharpoonup\Gamma \varphi$ in $C([0,T]; H^\frac{5}{4})$ if $\varphi_n\rightharpoonup \varphi$ in $L^2(0,T; H)$. It is enough to show that the set $\{\Gamma\varphi_n\}_{n\geq 1}$ is pre-compact and equicontinuous in $H^{\frac{5}{4}}$.

Introducing the operator
$$Q_N=I-P_N,$$
where $P_N$ is the finite-dimensional mapping from $H^\frac{5}{4}$ into $H^N=span\{e_i, i=1,\cdots, N\}$ and the sequence $\{e_j\}_{j\geq 1}$  is the basis of $H^\frac{5}{4}$. Since the sequence $\varphi_n\in L^2(0,T;H)$ has uniform bound of $n$,
\begin{align}\label{2.4*}
\|Q_N\Gamma\varphi_n\|_{H^\frac{5}{4}}&=\left\|Q_N\int_{0}^{t}G\varphi_nds\right\|_{H^\frac{5}{4}}\leq \int_{0}^{t}\|Q_NG\varphi_n\|_{H^\frac{5}{4}}ds\nonumber\\
&\leq \|Q_NG\|_{H^\frac{5}{4}}\int_{0}^{t}\|\varphi_n\|_{H}ds\nonumber\\
&\leq T\|Q_NG\|_{H^\frac{5}{4}}\|\varphi_n\|_{L^2(0,T;H)}\nonumber\\
&\leq CT\|Q_NG\|_{H^\frac{5}{4}}.
\end{align}
Since the right hand side term of \eqref{2.4*} goes to zero, then for any $\delta>0$, there exists $N_0$ such that for all $N>N_0$
\begin{align}\label{2.5*}\|Q_N\Gamma\varphi_n\|_{H^\frac{5}{4}}\leq \delta.\end{align}
Moreover, $P_N\Gamma\varphi_n$ lies in a finite-dimensional space which is bounded, and hence it is pre-compact. Combining \eqref{2.5*}, we could infer for every $\delta>0$, the sequence $\Gamma\varphi_n$ has a finite open cover of radius $\delta$ in $H^\frac{5}{4}$, hence pre-compact. From the formulation \eqref{2.3*} and the uniformly bounded of $\varphi_n$, we have the equicontinuity of $\Gamma\varphi_n$ in $H^\frac{5}{4}$. Consequently, we deduce that  $\Gamma\varphi_n$ is pre-compact in $C([0,T]; H^\frac{5}{4})$. This completes the proof.
\end{proof}
With the property in hands, we show the condition (1) in Theorem \ref{the4.1}.

\begin{proposition}\label{pro4.2} For any fixed $M>0$, $\varphi_\varepsilon, \varphi\in \mathcal{A}_M$ with $\varphi_\varepsilon\rightharpoonup \varphi$ in $L^2(0,T;H)$. Then the solution $\mathbf{v}^\varepsilon$ of equations (\ref{e3}) converges in distribution in $C([0,T]; H^\frac{5}{4})\cap L^2(0,T; H^\frac{5}{2})$ to the solution $\mathbf{v}$ of equations (\ref{e2.2}) as $\varepsilon\rightarrow 0$, that is, the process
\begin{eqnarray*}
\mathcal G^{\varepsilon}\left(\sqrt{\varepsilon}W+\int_{0}^{\cdot}\varphi_{\varepsilon}(s)ds\right)\rightarrow\mathcal G^{0}\left(\int_{0}^{\cdot}\varphi ds\right)
\end{eqnarray*}
in distribution in $C([0,T]; H^\frac{5}{4})\cap L^2(0,T; H^\frac{5}{2})$ as $\varepsilon\rightarrow 0$, where the solution mapping $$\mathcal{G}^{0}: C([0,T];H)\rightarrow C([0,T]; H^\frac{5}{4})\cap L^2(0,T; H^\frac{5}{2})$$ is defined by $$\mathcal{G}^{0}(g)=\mathbf{v},$$ for $g=\int_{0}^{\cdot}\varphi(s)ds\in C([0,T];H)$; otherwise, setting $\mathcal{G}^{0}(g)=0$.
\end{proposition}
\begin{proof} 

Applying the It\^{o} formula to $\frac{1}{2}\|\Lambda^\frac{5}{4}(\mathbf{v}^\varepsilon-\mathbf{v})\|_{L^2}^2$, we see
\begin{align}\label{2.7*}
&\frac{1}{2}d\|\Lambda^\frac{5}{4}(\mathbf{v}^\varepsilon-\mathbf{v})\|_{L^2}^2
+\|\Lambda^\frac{5}{2}(\mathbf{v}^\varepsilon-\mathbf{v})\|_{L^2}^2dt
\nonumber\\&=(\Lambda^\frac{5}{4}(G\varphi_\varepsilon-G\varphi), \Lambda^\frac{5}{4}(\mathbf{v}^\varepsilon-\mathbf{v}))dt
+\sqrt{\varepsilon}(\Lambda^\frac{5}{4}GdW,\Lambda^\frac{5}{4}(\mathbf{v}^\varepsilon-\mathbf{v}))
+\frac{1}{2}\varepsilon\|G\|^2_{L_2(H; H^\frac{5}{4})}dt.
\end{align}
We first deal with the control term
\begin{align}
&\left(\Lambda^\frac{5}{4}(G\varphi_\varepsilon-G\varphi), \Lambda^\frac{5}{4}(\mathbf{v}^\varepsilon-\mathbf{v})\right)=\left(\Lambda^\frac{5}{4}\frac{d}{dt}(\Gamma\varphi_\varepsilon-\Gamma\varphi), \Lambda^\frac{5}{4}(\mathbf{v}^\varepsilon-\mathbf{v})\right)\nonumber\\
&=\frac{d}{dt}\left(\Lambda^\frac{5}{4}(\Gamma\varphi_\varepsilon-\Gamma\varphi), \Lambda^\frac{5}{4}(\mathbf{v}^\varepsilon-\mathbf{v})\right)-\left(\Lambda^\frac{5}{4}(\Gamma\varphi_\varepsilon-\Gamma\varphi), \frac{d}{dt}\Lambda^\frac{5}{4}(\mathbf{v}^\varepsilon-\mathbf{v})\right)\nonumber\\
&=\frac{d}{dt}\left(\Lambda^\frac{5}{4}(\Gamma\varphi_\varepsilon-\Gamma\varphi), \Lambda^\frac{5}{4}(\mathbf{v}^\varepsilon-\mathbf{v})\right)\nonumber\\
&\quad-\left(\Lambda^\frac{5}{4}(\Gamma\varphi_\varepsilon-\Gamma\varphi), -\Lambda^\frac{5}{4}A^\frac{5}{4}(\mathbf{v}^\varepsilon-\mathbf{v})+\Lambda^\frac{5}{4}(G\varphi_\varepsilon-G\varphi)
+\sqrt{\varepsilon}\Lambda^\frac{5}{4}G\frac{dW}{dt}\right).
\end{align}

By the H\"{o}lder inequality, we obtain
\begin{align}
-\left(\Lambda^\frac{5}{4}(\Gamma\varphi_\varepsilon-\Gamma\varphi), -\Lambda^\frac{5}{4}A^\frac{5}{4}(\mathbf{v}^\varepsilon-\mathbf{v})\right)\leq \|\Gamma\varphi_\varepsilon-\Gamma\varphi\|_{H^\frac{5}{4}}\|\Lambda^\frac{5}{2}A^\frac{5}{4}(\mathbf{v}^\varepsilon-\mathbf{v})\|_{L^2},
\end{align}
and
\begin{align}\label{2.10*}
-\left(\Lambda^\frac{5}{4}(\Gamma\varphi_\varepsilon-\Gamma\varphi),\Lambda^\frac{5}{4}(G\varphi_\varepsilon-G\varphi)\right)\leq \|\Gamma\varphi_\varepsilon-\Gamma\varphi\|_{H^\frac{5}{4}}\|G\|_{L_2(H;H^\frac{5}{4})}\|\varphi_\varepsilon-\varphi\|_{H}.
\end{align}
Integrating of $t$, taking supremum and expectation, we have from \eqref{2.7*}-\eqref{2.10*}
\begin{align}\label{2.11*}
&\mathrm{E}\sup_{t\in [0,T]}\|\Lambda^\frac{5}{4}(\mathbf{v}^\varepsilon-\mathbf{v})\|_{L^2}^2
+\mathrm{E}\int_{0}^{T}\|\Lambda^\frac{5}{2}(\mathbf{v}^\varepsilon-\mathbf{v})\|_{L^2}^2dt\nonumber\\
&\leq \mathrm{E}\sup_{t\in [0,T]}\left|\left(\Lambda^\frac{5}{4}(\Gamma\varphi_\varepsilon-\Gamma\varphi), \Lambda^\frac{5}{4}(\mathbf{v}^\varepsilon-\mathbf{v})\right)\right|\nonumber\\
&\quad+C
\mathrm{E}\int_{0}^{T}\|\Lambda^\frac{5}{4}A^\frac{5}{4}(\mathbf{v}^\varepsilon-\mathbf{v})\|_{L^2}^2dt
\int_{0}^{T}\|\Gamma\varphi_\varepsilon-\Gamma\varphi\|_{H^\frac{5}{4}}^2dt\nonumber\\
&\quad+\mathrm{E}\int_{0}^{T}\|\Gamma\varphi_\varepsilon-\Gamma\varphi\|_{H^\frac{5}{4}}\|G\|_{L_2(H;H^\frac{5}{4})}\|\varphi_\varepsilon-\varphi\|_{H}dt\nonumber\\
&\quad+\mathrm{E}\sup_{t\in [0,T]}\left|\int_{0}^{t}\sqrt{\varepsilon}(\Lambda^\frac{5}{4}GdW,\Lambda^\frac{5}{4}(\Gamma\varphi_\varepsilon-\Gamma\varphi))\right|\nonumber\\
&\quad+\mathrm{E}\sup_{t\in [0,T]}\left|\int_{0}^{t}\sqrt{\varepsilon}(\Lambda^\frac{5}{4}GdW,\Lambda^\frac{5}{4}(\mathbf{v}^\varepsilon-\mathbf{v}))\right|\nonumber\\
&\quad+\frac{1}{2}\varepsilon T\|G\|^2_{L_2(H; H^\frac{5}{4})}.
\end{align}
For the first term in the right hand side of \eqref{2.11*}, we have
\begin{align}\label{4.12}
&\mathrm{E}\sup_{t\in [0,T]}\left|\left(\Lambda^\frac{5}{4}(\Gamma\varphi_\varepsilon-\Gamma\varphi), \Lambda^\frac{5}{4}(\mathbf{v}^\varepsilon-\mathbf{v})\right)\right|\nonumber\\
&\leq C\mathrm{E}\|\Gamma\varphi_\varepsilon-\Gamma\varphi\|^2_{C([0,T];H^\frac{5}{4})}+\frac{1}{2}\mathrm{E}\sup_{t\in [0,T]}\|\Lambda^\frac{5}{4}(\mathbf{v}^\varepsilon-\mathbf{v})\|_{L^2}^2.
\end{align}
And, we have
\begin{align}\label{4.13}
C\mathrm{E}\int_{0}^{T}\|\Gamma\varphi_\varepsilon-\Gamma\varphi\|_{H^\frac{5}{4}}^2dt\leq C\mathrm{E}\|\Gamma\varphi_\varepsilon-\Gamma\varphi\|^2_{C([0,T];H^\frac{5}{4})}.
\end{align}
Again, the H\"{o}lder inequality gives
\begin{align}
&\mathrm{E}\int_{0}^{T}\|\Gamma\varphi_\varepsilon-\Gamma\varphi\|_{H^\frac{5}{4}}\|G\|_{L_2(H;H^\frac{5}{4})}\|\varphi_\varepsilon-\varphi\|_{H}dt\nonumber\\
&\leq C\|G\|_{L_2(H;H^\frac{5}{4})}\mathrm{E}\|\Gamma\varphi_\varepsilon-\Gamma\varphi\|_{C([0,T];H^\frac{5}{4})}
\int_{0}^{T}\|\varphi_\varepsilon-\varphi\|^2_{H}dt.
\end{align}

For the martingale part,
\begin{align}
&\mathrm{E}\sup_{t\in [0,T]}\left|\int_{0}^{t}\sqrt{\varepsilon}(\Lambda^\frac{5}{4}GdW,\Lambda^\frac{5}{4}(\mathbf{v}^\varepsilon-\mathbf{v}))\right|\nonumber\\
&\leq \sqrt{\varepsilon}\mathrm{E}\left(\int_{0}^{T}\sum_{k\in \mathbb{Z}_0^3}(\Lambda^\frac{5}{4}Ge_k,\Lambda^\frac{5}{4}(\mathbf{v}^\varepsilon-\mathbf{v}))^2dt\right)^\frac{1}{2}\nonumber\\
&\leq \sqrt{\varepsilon}\mathrm{E}\left(\int_{0}^{T}\|\Lambda^\frac{5}{4}(\mathbf{v}^\varepsilon-\mathbf{v})
\|_{L^2}^2\|G\|_{L_2(H;H^\frac{5}{4})}^2dt\right)^\frac{1}{2}\nonumber\\
&\leq  \sqrt{\varepsilon}\|G\|_{L_2(H;H^\frac{5}{4})}\mathrm{E}\left(\int_{0}^{T}\|\Lambda^\frac{5}{4}(\mathbf{v}^\varepsilon-\mathbf{v})
\|_{L^2}^2dt\right)^\frac{1}{2}.
\end{align}
Similarly,
\begin{align}\label{2.16*}
&\mathrm{E}\sup_{t\in [0,T]}\left|\int_{0}^{t}\sqrt{\varepsilon}(\Lambda^\frac{5}{4}GdW,\Lambda^\frac{5}{4}(\Gamma\varphi_\varepsilon-\Gamma\varphi))\right|\nonumber\\
&\leq \sqrt{\varepsilon}\mathrm{E}\left(\int_{0}^{T}\sum_{k\in \mathbb{Z}_0^3}
(\Lambda^\frac{5}{4}Ge_k,\Lambda^\frac{5}{4}(\Gamma\varphi_\varepsilon-\Gamma\varphi))^2dt\right)^\frac{1}{2}\nonumber\\
&\leq  \sqrt{\varepsilon}\|G\|_{L_2(H;H^\frac{5}{4})}\mathrm{E}\left(\int_{0}^{T}\|\Lambda^\frac{5}{4}(\Gamma\varphi_\varepsilon-\Gamma\varphi)
\|_{L^2}^2dt\right)^\frac{1}{2}\nonumber\\
&\leq  C(T)\sqrt{\varepsilon}\|G\|_{L_2(H;H^\frac{5}{4})}
\mathrm{E}\left(\int_{0}^{T}\|G\|^2_{L_2(H;H^\frac{5}{4})}\|\varphi_\varepsilon-\varphi\|^2_{H}dt\right)^\frac{1}{2}\nonumber\\
&\leq C(M,T)\sqrt{\varepsilon}\|G\|^2_{L_2(H;H^\frac{5}{4})}.
\end{align}
Taking \eqref{4.3*} and \eqref{2.11*}-\eqref{2.16*} into account, we obtain
\begin{align}
&\mathrm{E}\sup_{t\in [0,T]}\|\Lambda^\frac{5}{4}(\mathbf{v}^\varepsilon-\mathbf{v})\|_{L^2}^2
+\mathrm{E}\int_{0}^{T}\|\Lambda^\frac{5}{2}(\mathbf{v}^\varepsilon-\mathbf{v})\|_{L^2}^2dt\nonumber\\
&\leq C\mathrm{E}\|\Gamma\varphi_\varepsilon-\Gamma\varphi\|^2_{C([0,T];H^\frac{5}{4})}
+C(M)\|G\|_{L_2(H;H^\frac{5}{4})}\mathrm{E}\|\Gamma\varphi_\varepsilon-\Gamma\varphi\|_{C([0,T];H^\frac{5}{4})}\nonumber\\
&\quad+\sqrt{\varepsilon}\|G\|_{L_2(H;H^\frac{5}{4})}\mathrm{E}\left(\int_{0}^{T}\|\Lambda^\frac{5}{4}(\mathbf{v}^\varepsilon-\mathbf{v})
\|_{L^2}^2dt\right)^\frac{1}{2}\nonumber\\
&\quad+C(M,T)\sqrt{\varepsilon}\|G\|^2_{L_2(H;H^\frac{5}{4})}+\frac{1}{2}\varepsilon T\|G\|^2_{L_2(H; H^\frac{5}{4})}.
\end{align}
Using Lemma \ref{lem2.1}, we conclude that as $\varepsilon\rightarrow 0$
\begin{align*}
\mathrm{E}\sup_{t\in [0,T]}\|\Lambda^\frac{5}{4}(\mathbf{v}^\varepsilon-\mathbf{v})\|_{L^2}^2
+\mathrm{E}\int_{0}^{T}\|\Lambda^\frac{5}{2}(\mathbf{v}^\varepsilon-\mathbf{v})\|_{L^2}^2dt\rightarrow 0,
\end{align*}
which implies $\mathbf{v}^\varepsilon\rightarrow \mathbf{v}$ in distribution in $C([0,T]; H^\frac{5}{4})\cap L^2(0,T; H^\frac{5}{2})$. We finish the proof.
\end{proof}
We next prove the convergence of solutions of \eqref{e2.2} with respect to $\varphi$ with $L^2(0,T; H)$-weak topology.
\begin{proposition}\label{pro2.2*} For every $M<\infty$, the set
\begin{eqnarray*}
K_{M}=\{\mathbf{v}_{\varphi}:\varphi\in S_{M}\}
\end{eqnarray*}
is a compact subset of $C([0,T]; H^\frac{5}{4})\cap L^2(0,T; H^\frac{5}{2})$.
\end{proposition}

\begin{proof} Denote by $\mathbf{v}_{\varphi_n}$ be the solutions to the deterministic controlled equations \eqref{e2.2} with controlled term $\int_0^tG\varphi_nds$, where the processes $\varphi_n\in S_M$. Since the set $S_M$ is a bound closed set, therefore there exists a subsequence of $\varphi_n$ still denoted by $\varphi_n$ converging to $\varphi$ weakly in $L^2(0,T; H)$. We next show that the solutions $\mathbf{v}_{\varphi_n}$ converge to $\mathbf{v}_\varphi$ in $C([0,T];H^\frac{5}{4})$. First,
\begin{align} \label{2.19*} d(\mathbf{v}_{\varphi_n}-\mathbf{v}_{\varphi})+A^\frac{5}{4}(\mathbf{v}_{\varphi_n}-\mathbf{v}_{\varphi})dt=G(\mathbf{v}_{\varphi_n}-\mathbf{v}_{\varphi})dt.
\end{align}
Applying $\Lambda^\frac{5}{4}$ on both sides of \eqref{2.19*}, then taking inner product with $\Lambda^\frac{5}{4}(\mathbf{v}_{\varphi_n}-\mathbf{v}_{\varphi})$, lead to
\begin{align}\label{2.20*}
d\|\Lambda^\frac{5}{4}(\mathbf{v}_{\varphi_n}-\mathbf{v}_{\varphi})\|_{L^2}^2+2\|\Lambda^\frac{5}{2}(\mathbf{v}_{\varphi_n}-\mathbf{v}_{\varphi})\|_{L^2}^2dt
=2(\Lambda^\frac{5}{4}G(\mathbf{v}_{\varphi_n}-\mathbf{v}_{\varphi}), \Lambda^\frac{5}{4}(\mathbf{v}_{\varphi_n}-\mathbf{v}_{\varphi}))dt.
\end{align}
Similar to the argument as that of in Proposition \ref{pro4.2}, we also have
\begin{align}
&\left(\Lambda^\frac{5}{4}(G\varphi_n-G\varphi), \Lambda^\frac{5}{4}(\mathbf{v}_{\varphi_n}-\mathbf{v}_\varphi)\right)=\left(\Lambda^\frac{5}{4}\frac{d}{dt}(\Gamma\varphi_n-\Gamma\varphi), \Lambda^\frac{5}{4}(\mathbf{v}_{\varphi_n}-\mathbf{v}_\varphi)\right)\nonumber\\
&=\frac{d}{dt}\left(\Lambda^\frac{5}{4}(\Gamma\varphi_n-\Gamma\varphi), \Lambda^\frac{5}{4}(\mathbf{v}_{\varphi_n}-\mathbf{v}_\varphi)\right)-\left(\Lambda^\frac{5}{4}(\Gamma\varphi_n-\Gamma\varphi), \frac{d}{dt}\Lambda^\frac{5}{4}(\mathbf{v}_{\varphi_n}-\mathbf{v}_\varphi)\right)\nonumber\\
&=\frac{d}{dt}\left(\Lambda^\frac{5}{4}(\Gamma\varphi_n-\Gamma\varphi),\Lambda^\frac{5}{4}(\mathbf{v}_{\varphi_n}-\mathbf{v}_\varphi)\right)\nonumber\\
&\quad-\left(\Lambda^\frac{5}{4}(\Gamma\varphi_n-\Gamma\varphi), -\Lambda^\frac{5}{4}A^\frac{5}{4}(\mathbf{v}_{\varphi_n}-\mathbf{v}_\varphi)+\Lambda^\frac{5}{4}(G\varphi_n-G\varphi)
\right).
\end{align}
As \eqref{4.12} and \eqref{4.13}, we have
\begin{align}
\left(\Lambda^\frac{5}{4}(\Gamma\varphi_n-\Gamma\varphi), -\Lambda^\frac{5}{4}A^\frac{5}{4}(\mathbf{v}_{\varphi_n}-\mathbf{v}_\varphi)\right)\leq \|\Gamma\varphi_n-\Gamma\varphi\|_{H^\frac{5}{4}}\|\Lambda^\frac{5}{4}A^\frac{5}{4}(\mathbf{v}_{\varphi_n}-\mathbf{v}_\varphi)\|_{L^2},
\end{align}
and
\begin{align}\label{2.23*}
(\Lambda^\frac{5}{4}(\Gamma\varphi_n-\Gamma\varphi),\Lambda^\frac{5}{4}(G\varphi_n-G\varphi))\leq \|\Gamma\varphi_n-\Gamma\varphi\|_{H^\frac{5}{4}}\|G\|_{L_2(H;H^\frac{5}{4})}\|\varphi_n-\varphi\|_{H}.
\end{align}
Integrating of $t$, taking supremum we have from \eqref{2.20*}-\eqref{2.23*}
\begin{align}\label{2.24*}
&\sup_{t\in [0,T]}\|\Lambda^\frac{5}{4}(\mathbf{v}_{\varphi_n}-\mathbf{v}_{\varphi})\|_{L^2}^2
+\int_{0}^{T}\|\Lambda^\frac{5}{2}(\mathbf{v}_{\varphi_n}-\mathbf{v}_{\varphi})\|_{L^2}^2dt\nonumber\\
&\leq \sup_{t\in [0,T]}\left(\Lambda^\frac{5}{4}(\Gamma\varphi_n-\Gamma\varphi),\Lambda^\frac{5}{4}(\mathbf{v}_{\varphi_n}-\mathbf{v}_\varphi)\right)\nonumber\\
&\quad+C\|\Gamma\varphi_n-\Gamma\varphi\|_{C([0,T];H^\frac{5}{4})}
\int_{0}^{T}\|\Lambda^\frac{5}{4}A^\frac{5}{4}(\mathbf{v}_{\varphi_n}-\mathbf{v}_\varphi)\|^2_{L^2}dt\nonumber\\
&\quad+C\|\Gamma\varphi_n-\Gamma\varphi\|_{C([0,T];H^\frac{5}{4})}\|G\|_{L_2(H;H^\frac{5}{4})}\int_{0}^{T}\|\varphi_n-\varphi\|_{H}dt\nonumber\\
&\leq C\|\Gamma\varphi_n-\Gamma\varphi\|_{C([0,T];H^\frac{5}{4})}^2+\frac{1}{2}\sup_{t\in [0,T]}\|\Lambda^\frac{5}{4}(\mathbf{v}_{\varphi_n}-\mathbf{v}_{\varphi})\|_{L^2}^2\nonumber\\
&\quad+C(M)\|\Gamma\varphi_n-\Gamma\varphi\|_{C([0,T];H^\frac{5}{4})}\|G\|_{L_2(H;H^\frac{5}{4})}.
\end{align}
The second term of right hand side of \eqref{2.24*} could be absorbed by the first term on the left hand side.  By Lemma \ref{lem2.1}, we know the right hand side terms converge to zero. We obtain the compactness. \end{proof}

Combining  Theorem \ref{the4.1}, Propositions \ref{pro4.2} and \ref{pro2.2*}, the solution of equations \eqref{e2.1} satisfies the large deviations in $C([0,T];H^\frac{5}{4})\cap L^2(0,T; H^\frac{5}{2})$ with good rate function
\begin{align*}
I(\mathbf{v})=\inf_{\left\{\varphi\in L^{2}(0,T;H):\mathbf{v}=\mathcal{G}^{0}(\int_{0}^{\cdot}\varphi(s)ds)\right\}}\left\{\frac{1}{2}\int_{0}^{T}\|\varphi\|_{H}^{2}dt\right\}.
\end{align*}
We complete the proof.

\section{Uniform large deviations}

In this section, we establish the uniform large deviations of the family of $\mathbf{u}^\varepsilon$ in $C([0,T];H)$. We begin with stating the definition of uniform large deviations.
\begin{definition}\label{def3.1} For each $x\in \mathcal{D}$, let $\{\mu^x_\varepsilon\}_{\varepsilon>0}$ be a family of probability measures on Banach space $\mathcal{E}$, we say the family $\{\mu^x_\varepsilon\}_{\varepsilon>0}$ satisfies the uniform large deviation principle with speed $\varepsilon$ and with good rate function $I^x: \mathcal{E}\rightarrow \infty$, if

i. for any $s\geq 0$, $\delta>0$ and $\gamma>0$, there exists $\mathcal{\varepsilon}_0$ such that for all $\varepsilon\leq \varepsilon_0$
$$\inf_{x\in \mathcal{D}}\left(\mu^x_\varepsilon(B_\mathcal{E}(h, \delta))-{\rm exp}\left(-\frac{I^x(h)+\gamma}{\varepsilon}\right)\right)\geq 0,$$
for any $h\in \mathcal{H}^x(s)$, where $\mathcal{H}^x(s):=\{h\in \mathcal{E}: I^x(h)\leq s\}$ and $B_\mathcal{E}(h, \delta):=\{l\in \mathcal{E}: \|l-h\|_{\mathcal{E}}<\delta\}$;

ii. for any $s\geq 0$, $\delta>0$ and $\gamma>0$, there exists $\mathcal{\varepsilon}_0$ such that for all $\varepsilon\leq \varepsilon_0$
$$\sup_{x\in \mathcal{D}}\mu^x_\varepsilon(B^c_\mathcal{E}(\mathcal{H}^x(s), \delta))\leq {\rm exp}\left(-\frac{s-\gamma}{\varepsilon}\right),$$
where $B^c_\mathcal{E}(\mathcal{H}^x(s), \delta)=\{h\in \mathcal{E}: dist_\mathcal{E}(h, \mathcal{H}^x(s))\geq \delta\}.$
\end{definition}

The proof of uniform large deviations is based on the following uniform contraction principle given by \cite[Theorem 3.3]{cp}.
\begin{proposition}\label{pro3.1} The family of composite measures $\mu^\varepsilon=\pi^\varepsilon\circ \mathfrak{M}$ satisfies the uniform large deviation principle in $\mathcal{E}$ with rate function $I$, with respect to $x$ in non-empty set $\mathcal{D}$, where
$$I(\varphi)=\{J(\psi):~ \psi\in F, ~\varphi\in \mathfrak{M}\psi\},$$
if the followings hold:

i. the family of measures $\pi^\varepsilon$ satisfies the large deviation principle in $F$ with rate function $J$;

ii. the family of measures $\pi^\varepsilon$ is exponential tight in $E$, where $E$ is a Banach space with the continuous embedding $F\subset E$, that is,
for each $s>0$, there exists $R_s$ and $\varepsilon_0>0$ such that
$$\pi^\varepsilon(B_E(R_s)\cap F)\geq 1-{\rm exp}\left(\frac{s}{\varepsilon}\right),$$
for any $\varepsilon\leq \varepsilon_0$;

iii. the map $\mathfrak{M}:F\rightarrow \mathcal{E}$ is Lipschitz continuous on the ball of $E$, uniformly in $x\in \mathcal{D}$.
\end{proposition}
\begin{remark} The uniform large deviation principle usually follows from the Lipschitz continuity of map $\mathfrak{M}$. Generally, for the stochastic partial differential equations, we fail to show the global Lipschitz continuity due to the nonlinearity construction. Therefore, the exponential tight will be used for compensating the drawbacks. Intuitively, the exponential tight implies the family of solutions $\mathbf{u}^\varepsilon$ located in a ball of $E$ with any arbitrary large probability, which means the Lipschitz continuous of map $\mathfrak{M}$ could be considered to be global with probability "1" as $\varepsilon\rightarrow 0$.
\end{remark}

Inspired by Proposition \ref{pro3.1}, we decompose the original equations \eqref{e1*} into two equations: one of is a linear stochastic differential equations \eqref{e2.1}, while the another is the following random Navier-Stokes equations:
\begin{eqnarray}
\left\{\begin{array}{ll}
\!\!\!d\widetilde{\mathbf{u}}+A^\frac{5}{4}\widetilde{\mathbf{u}}dt+B(\widetilde{\mathbf{u}}+\mathbf{v},\widetilde{\mathbf{u}}+\mathbf{v})dt=0,\\
\!\!\!\widetilde{\mathbf{u}}|_{t=0}=\mathbf{u}_0.\\
\end{array}\right.
\end{eqnarray}
Denote by $\widetilde{\mathbf{u}}$ be the solution of above equations, and define the mapping $\mathfrak{M}: \mathbf{v}\in L^4(0,T; H^\frac{5}{4})\rightarrow \widetilde{\mathbf{u}}\in C([0,T]; H)$, note that ${\rm I}+\mathfrak{M}$ is a mapping from $\mathbf{v}$ into $\mathbf{u}^\varepsilon$, where $\mathbf{u}^\varepsilon$ is the solution of equations \eqref{e1*}.

In the previous section, we already established the large deviations of the family of $\{\mathbf{v}^\varepsilon\}_{\varepsilon>0}$ in space $C([0,T]; H^\frac{5}{4})\cap L^2(0,T; H^\frac{5}{2})$. According to Proposition \ref{pro3.1}, it remains to show conditions ii and iii. In our setting, we choose $F=C([0,T]; H^\frac{5}{4})$, $E=L^4(0,T; H^\frac{5}{4})$ and $\mathcal{E}=C([0,T]; H)$.

\begin{lemma}\label{lem3.3} For any $R>0$ and $\mathbf{u}_0\in B_R(H)$, the mapping $\mathfrak{M}$ is Lipschitz continuous on a ball $B_{L^4(0,T;H^\frac{5}{4})}(R)$.
\end{lemma}
\begin{proof} Taking inner product with $\widetilde{\mathbf{u}}$, we have
\begin{align*}
d\|\widetilde{\mathbf{u}}\|_{H}^2+2\|\Lambda^\frac{5}{4}\widetilde{\mathbf{u}}\|_{L^2}^2dt\leq 2|(B(\widetilde{\mathbf{u}}+\mathbf{v}, \mathbf{v}),\widetilde{\mathbf{u}}) |dt.
\end{align*}
By the H\"{o}lder inequality and embedding $H^\frac{5}{4}\rightarrow L^{12}$, $H^\frac{5}{4}\rightarrow H^{1, \frac{12}{5}}$ again, we obtain
\begin{align}\label{5.3}
|(B(\widetilde{\mathbf{u}}+\mathbf{v}, \mathbf{v}),\widetilde{\mathbf{u}}) |&\leq \|\Lambda\widetilde{\mathbf{u}}\|^2_{L^\frac{12}{5}}+C\|\widetilde{\mathbf{u}}+\mathbf{v}\|_{H}^2\|\mathbf{v}\|^2_{L^{12}}\nonumber\\
&\leq \|\Lambda^\frac{5}{4}\widetilde{\mathbf{u}}\|_{L^2}^2+C\|\widetilde{\mathbf{u}}+\mathbf{v}\|_{H}^2\|\Lambda^\frac{5}{4}\mathbf{v}\|^2_{L^{2}}.
\end{align}
Following from the Gronwall lemma
\begin{align}\label{1.26*}
\sup_{t\in [0,T]}\|\widetilde{\mathbf{u}}\|_{H}^2+\int_{0}^{T}\|\Lambda^\frac{5}{4}\widetilde{\mathbf{u}}\|_{L^2}^2dt\leq C\left(\|\widetilde{\mathbf{u}}_0\|_{H}^2+\int_{0}^{T}\|\Lambda^\frac{5}{4}\mathbf{v}\|^4_{L^{2}}dt\right){\rm exp}\left(\int_{0}^{T}\|\Lambda^\frac{5}{4}\mathbf{v}\|^2_{L^{2}}dt\right).
\end{align}

For any $\mathbf{v}_1, \mathbf{v}_2\in B_{L^4(0,T;L^2)}(R)$, the mapping $\widetilde{\mathbf{u}}_1-\widetilde{\mathbf{u}}_2$ satisfies
\begin{align}\label{3.15}
&d(\widetilde{\mathbf{u}}_1-\widetilde{\mathbf{u}}_2)+A^\alpha(\widetilde{\mathbf{u}}_1-\widetilde{\mathbf{u}}_2)dt\nonumber\\
&=-B(\widetilde{\mathbf{u}}_1-\widetilde{\mathbf{u}}_2+\mathbf{v}_1-\mathbf{v}_2, \widetilde{\mathbf{u}}_1+\mathbf{v}_1)dt\nonumber\\
&\quad-B(\widetilde{\mathbf{u}}_2+\mathbf{v}_2, \widetilde{\mathbf{u}}_1-\widetilde{\mathbf{u}}_2+\mathbf{v}_1-\mathbf{v}_2)dt.
\end{align}
Taking inner product with $\widetilde{\mathbf{u}}_1-\widetilde{\mathbf{u}}_2$ in \eqref{3.15}, by cancellation property \eqref{2.1} again
\begin{align}\label{4.6**}
&\frac{1}{2}d\|\widetilde{\mathbf{u}}_1-\widetilde{\mathbf{u}}_2\|_{H}^2
+\|\Lambda^\frac{5}{4}(\widetilde{\mathbf{u}}_1-\widetilde{\mathbf{u}}_2)\|_{L^2}^2dt\nonumber\\
&=-(B(\widetilde{\mathbf{u}}_1-\widetilde{\mathbf{u}}_2+\mathbf{v}_1-\mathbf{v}_2, \widetilde{\mathbf{u}}_1+\mathbf{v}_1), \widetilde{\mathbf{u}}_1-\widetilde{\mathbf{u}}_2)dt\nonumber\\
&\quad-(B(\widetilde{\mathbf{u}}_2+\mathbf{v}_2, \mathbf{v}_1-\mathbf{v}_2), \widetilde{\mathbf{u}}_1-\widetilde{\mathbf{u}}_2)dt.
\end{align}
As \eqref{5.3}, we have
\begin{align}
&|-(B(\widetilde{\mathbf{u}}_1-\widetilde{\mathbf{u}}_2+\mathbf{v}_1-\mathbf{v}_2, \widetilde{\mathbf{u}}_1+\mathbf{v}_1), \widetilde{\mathbf{u}}_1-\widetilde{\mathbf{u}}_2)|\nonumber\\
&\leq \|\Lambda(\widetilde{\mathbf{u}}_1-\widetilde{\mathbf{u}}_2)\|_{L^\frac{12}{5}}\|\widetilde{\mathbf{u}}_1+\mathbf{v}_1\|_{L^{12}}
\|\widetilde{\mathbf{u}}_1-\widetilde{\mathbf{u}}_2\|_{H}\nonumber\\
&\quad+\|\Lambda(\widetilde{\mathbf{u}}_1-\widetilde{\mathbf{u}}_2)\|_{L^\frac{12}{5}}\|\widetilde{\mathbf{u}}_1+\mathbf{v}_1\|_{H}
\|\mathbf{v}_1-\mathbf{v}_2\|_{L^{12}}\nonumber\\
&\leq \frac{1}{4}\|\Lambda^\frac{5}{4}(\widetilde{\mathbf{u}}_1-\widetilde{\mathbf{u}}_2)\|^2_{L^2}+C\|\widetilde{\mathbf{u}}_1-\widetilde{\mathbf{u}}_2\|^2_{H}
\|\Lambda^\frac{5}{4}(\widetilde{\mathbf{u}}_1+\mathbf{v}_1)\|^2_{L^{2}}\nonumber\\
&\quad+C\|\Lambda^\frac{5}{4}(\mathbf{v}_1-\mathbf{v}_2)\|^2_{L^2}
\|\widetilde{\mathbf{u}}_1+\mathbf{v}_1\|^2_{H},
\end{align}
and
\begin{align}\label{4.8}
&|-(B(\widetilde{\mathbf{u}}_2+\mathbf{v}_2, \mathbf{v}_1-\mathbf{v}_2), \widetilde{\mathbf{u}}_1-\widetilde{\mathbf{u}}_2)|\nonumber\\
&\leq \|\Lambda(\widetilde{\mathbf{u}}_1-\widetilde{\mathbf{u}}_2)\|_{L^\frac{12}{5}}\|\widetilde{\mathbf{u}}_2+\mathbf{v}_2\|_{H}
\|\mathbf{v}_1-\mathbf{v}_2\|_{L^{12}}\nonumber\\
&\leq \frac{1}{4}\|\Lambda^\frac{5}{4}(\widetilde{\mathbf{u}}_1-\widetilde{\mathbf{u}}_2)\|^2_{L^2}+C\|\Lambda^\frac{5}{4}(\mathbf{v}_1-\mathbf{v}_2)\|^2_{L^2}
\|\widetilde{\mathbf{u}}_2+\mathbf{v}_2\|^2_{H}.
\end{align}
Combining \eqref{4.6**}-\eqref{4.8}, we have
\begin{align}\label{3.16**}
&\frac{1}{2}d\|\widetilde{\mathbf{u}}_1-\widetilde{\mathbf{u}}_2\|_{H}^2+
\|\Lambda^\frac{5}{4}(\widetilde{\mathbf{u}}_1-\widetilde{\mathbf{u}}_2)\|_{L^2}^2dt\nonumber\\
&\leq C\|\widetilde{\mathbf{u}}_1-\widetilde{\mathbf{u}}_2\|^2_{H}
\|\Lambda^\frac{5}{4}(\widetilde{\mathbf{u}}_1+\mathbf{v}_1)\|^2_{L^{2}}dt+C\|\Lambda^\frac{5}{4}(\mathbf{v}_1-\mathbf{v}_2)\|^2_{L^2}
\|\widetilde{\mathbf{u}}_1+\mathbf{v}_1\|^2_{H}dt.
\end{align}
Integrating of $t$ in \eqref{3.16**}, using \eqref{1.26*} and the Gronwall lemma, we see
\begin{align*}
&\sup_{t\in [0,T]}\|\widetilde{\mathbf{u}}_1-\widetilde{\mathbf{u}}_2\|_{H}^2
+\int_{0}^{T}\|\Lambda^\frac{5}{4}(\widetilde{\mathbf{u}}_1-\widetilde{\mathbf{u}}_2)\|_{L^2}^2dt\nonumber\\
&\leq C{\rm exp}\left(\int_{0}^{T}\|\Lambda^\frac{5}{4}(\widetilde{\mathbf{u}}_1+\mathbf{v}_1)\|^2_{L^{2}}dt\right)
\int_{0}^{T}\|\Lambda^\frac{5}{4}(\mathbf{v}_1-\mathbf{v}_2)\|^2_{L^2}
\|\widetilde{\mathbf{u}}_1+\mathbf{v}_1\|^2_{H}dt\nonumber\\
&\leq {\rm exp}\left(\int_{0}^{T}\|\Lambda^\frac{5}{4}(\widetilde{\mathbf{u}}_1+\mathbf{v}_1)\|^2_{L^{2}}dt\right)\|\Lambda^\frac{5}{4}(\mathbf{v}_1-\mathbf{v}_2)\|^2_{L^4(0,T; L^2)}\left(\|\widetilde{\mathbf{u}}_1\|_{L^\infty(0,T; H)}+\|\mathbf{v}_1\|_{L^4(0,T;H^\frac{5}{4})}\right)\nonumber\\
&\leq C(R)\|\Lambda^\frac{5}{4}(\mathbf{v}_1-\mathbf{v}_2)\|^2_{L^4(0,T; L^2)}.
\end{align*}
This completes the proof.
\end{proof}

We next verify the exponential tight condition.
\begin{lemma}\label{lem5.2} The law of family of  $\mathbf{v}^\varepsilon$ is exponential tight in $L^4(0,T; H^\frac{5}{4})$, that is, for any $s>0$, there exists $\varepsilon_0>0$ and $R_s>0$ such that for any $\varepsilon<\varepsilon_0$
$$\mathcal{L}^\varepsilon\left(B_{L^4(0,T;H^\frac{5}{4})}(R_s)\right)\geq 1-{\rm exp}\left(-\frac{s}{\varepsilon}\right).$$
\end{lemma}
\begin{proof} Introducing the functions
$$f(x)=(1+x)^\frac{1}{4},$$
and
$$F(x)={\rm exp}\left(\frac{f(x)}{\varepsilon}\right).$$
Then, we have
$$DF(x)=\frac{F(x)Df}{\varepsilon}=\frac{F(x)f^{-3}(x)}{4\varepsilon},$$
and
$$D^2F(x)=\frac{F(x)Df\times Df}{\varepsilon^2}+\frac{F(x)D^2f}{\varepsilon}=\frac{F(x)f^{-6}(x)}{16\varepsilon^2}-\frac{3F(x)f^{-7}(x)}{16\varepsilon}.$$
Using the It\^{o} formula to $F(x)$ with $x=\|\Lambda^\frac{5}{4}\mathbf{v}^\varepsilon\|_{L^2}^4$, we have
\begin{align}\label{5.9*}
&dF(\|\Lambda^\frac{5}{4}\mathbf{v}^\varepsilon\|_{L^2}^4)
=\frac{F(\|\Lambda^\frac{5}{4}\mathbf{v}^\varepsilon\|_{L^2}^4)f^{-3}(\|\Lambda^\frac{5}{4}\mathbf{v}^\varepsilon\|_{L^2}^4)}{4\varepsilon}
d\|\Lambda^\frac{5}{4}\mathbf{v}^\varepsilon\|_{L^2}^4\nonumber\\
&+\frac{1}{2}\left(\frac{F(\|\Lambda^\frac{5}{4}\mathbf{v}^\varepsilon\|_{L^2}^4)f^{-6}(\|\Lambda^\frac{5}{4}\mathbf{v}^\varepsilon\|_{L^2}^4)}{16\varepsilon^2}
-\frac{3F(\|\Lambda^\frac{5}{4}\mathbf{v}^\varepsilon\|_{L^2}^4)f^{-7}(\|\Lambda^\frac{5}{4}\mathbf{v}^\varepsilon\|_{L^2}^4)}{16\varepsilon}\right)
d\langle \!\langle\|\Lambda^\frac{5}{4}\mathbf{v}^\varepsilon\|_{L^2}^4,\|\Lambda^\frac{5}{4}\mathbf{v}^\varepsilon\|_{L^2}^4\rangle \!\rangle,
\end{align}
where notation $\langle\!\langle\cdot, \cdot\rangle \!\rangle$ stands for the quadratic variation.

Using the It\^{o} formula to $\left(\|\Lambda^\frac{5}{4}\mathbf{v}^\varepsilon\|_{L^2}^2\right)^2$, we have
\begin{align}\label{5.11}
d\left(\|\Lambda^\frac{5}{4}\mathbf{v}^\varepsilon\|_{L^2}^2\right)^2=2\|\Lambda^\frac{5}{4}\mathbf{v}^\varepsilon\|_{L^2}^2
d\|\Lambda^\frac{5}{4}\mathbf{v}^\varepsilon\|_{L^2}^2+d\langle\!\langle\|\Lambda^\frac{5}{4}\mathbf{v}^\varepsilon\|_{L^2}^2, \|\Lambda^\frac{5}{4}\mathbf{v}^\varepsilon\|_{L^2}^2\rangle\!\rangle.
\end{align}
Since
\begin{align}\label{5.12}
d\|\Lambda^\frac{5}{4}\mathbf{v}^\varepsilon\|_{L^2}^2=-2\|\Lambda^\frac{5}{2}\mathbf{v}^\varepsilon\|_{L^2}^2dt+2\sqrt{\varepsilon}(\Lambda^\frac{5}{4}G, \Lambda^\frac{5}{4}\mathbf{v}^\varepsilon)dW+\varepsilon\|G\|^2_{L_2(H; H^\frac{5}{4})}dt,
\end{align}
we have by \eqref{5.11} and \eqref{5.12}
\begin{align}\label{5.13}
d\left(\|\Lambda^\frac{5}{4}\mathbf{v}^\varepsilon\|_{L^2}^2\right)^2
&=-4\|\Lambda^\frac{5}{2}\mathbf{v}^\varepsilon\|_{L^2}^2\|\Lambda^\frac{5}{4}\mathbf{v}^\varepsilon\|_{L^2}^2
dt+4\sqrt{\varepsilon}\|\Lambda^\frac{5}{4}\mathbf{v}^\varepsilon\|_{L^2}^2(\Lambda^\frac{5}{4}G, \Lambda^\frac{5}{4}\mathbf{v}^\varepsilon)dW\nonumber\\
&\quad+2\varepsilon\|\Lambda^\frac{5}{4}\mathbf{v}^\varepsilon\|_{L^2}^2\|G\|^2_{L_2(H; H^\frac{5}{4})}dt+4\varepsilon(\Lambda^\frac{5}{4}G, \Lambda^\frac{5}{4}\mathbf{v}^\varepsilon)^2dt,
\end{align}
and
\begin{align}\label{5.14}
d\langle\!\langle\|\Lambda^\frac{5}{4}\mathbf{v}^\varepsilon\|_{L^2}^4,\|\Lambda^\frac{5}{4}\mathbf{v}^\varepsilon\|_{L^2}^4\rangle\!\rangle
&=16\varepsilon\|\Lambda^\frac{5}{4}\mathbf{v}^\varepsilon\|_{L^2}^4(\Lambda^\frac{5}{4}G, \Lambda^\frac{5}{4}\mathbf{v}^\varepsilon)^2dt\nonumber\\
&\leq 16\varepsilon f^4(\|\Lambda^\frac{5}{4}\mathbf{v}^\varepsilon\|_{L^2}^4)(\Lambda^\frac{5}{4}G, \Lambda^\frac{5}{4}\mathbf{v}^\varepsilon)^2dt.
\end{align}
Therefore, from \eqref{5.13} and \eqref{5.14}, we get
\begin{align}\label{5.15}
&\frac{F(\|\Lambda^\frac{5}{4}\mathbf{v}^\varepsilon\|_{L^2}^4)f^{-3}(\|\Lambda^\frac{5}{4}\mathbf{v}^\varepsilon\|_{L^2}^4)}{4\varepsilon}
d\|\Lambda^\frac{5}{4}\mathbf{v}^\varepsilon\|_{L^2}^4\nonumber\\
&\leq -\frac{F(\|\Lambda^\frac{5}{4}\mathbf{v}^\varepsilon\|_{L^2}^4)f^{-3}(\|\Lambda^\frac{5}{4}\mathbf{v}^\varepsilon\|_{L^2}^4)}{\varepsilon}
\|\Lambda^\frac{5}{4}\mathbf{v}^\varepsilon\|_{L^2}^4dt\nonumber\\
&\quad+\frac{F(\|\Lambda^\frac{5}{4}\mathbf{v}^\varepsilon\|_{L^2}^4)f^{-3}(\|\Lambda^\frac{5}{4}\mathbf{v}^\varepsilon\|_{L^2}^4)}{\varepsilon}
\sqrt{\varepsilon}\|\Lambda^\frac{5}{4}\mathbf{v}^\varepsilon\|_{L^2}^2(\Lambda^\frac{5}{4}G, \Lambda^\frac{5}{4}\mathbf{v}^\varepsilon)dW\nonumber\\
&\quad+2\varepsilon\frac{F(\|\Lambda^\frac{5}{4}\mathbf{v}^\varepsilon\|_{L^2}^4)f^{-3}(\|\Lambda^\frac{5}{4}\mathbf{v}^\varepsilon\|_{L^2}^4)}{4\varepsilon}
\|\Lambda^\frac{5}{4}\mathbf{v}^\varepsilon\|_{L^2}^2\|G\|^2_{L_2(H; H^\frac{5}{4})}dt\nonumber\\
&\quad+4\varepsilon\frac{F(\|\Lambda^\frac{5}{4}\mathbf{v}^\varepsilon\|_{L^2}^4)f^{-3}(\|\Lambda^\frac{5}{4}\mathbf{v}^\varepsilon\|_{L^2}^4)}{4\varepsilon}(\Lambda^\frac{5}{4}G, \Lambda^\frac{5}{4}\mathbf{v}^\varepsilon)^2dt,
\end{align}
and
\begin{align}
&\frac{1}{2}\left(\frac{F(\|\Lambda^\frac{5}{4}\mathbf{v}^\varepsilon\|_{L^2}^4)f^{-6}(\|\Lambda^\frac{5}{4}\mathbf{v}^\varepsilon\|_{L^2}^4)}{16\varepsilon^2}
-\frac{3F(\|\Lambda^\frac{5}{4}\mathbf{v}^\varepsilon\|_{L^2}^4)f^{-7}(\|\Lambda^\frac{5}{4}\mathbf{v}^\varepsilon\|_{L^2}^4)}{16\varepsilon}\right)
d\langle\!\langle\|\Lambda^\frac{5}{4}\mathbf{v}^\varepsilon\|_{L^2}^4,\|\Lambda^\frac{5}{4}\mathbf{v}^\varepsilon\|_{L^2}^4\rangle\!\rangle\nonumber\\
&\leq \frac{\varepsilon}{2}\left(\frac{F(\|\Lambda^\frac{5}{4}\mathbf{v}^\varepsilon\|_{L^2}^4)
f^{-2}(\|\Lambda^\frac{5}{4}\mathbf{v}^\varepsilon\|_{L^2}^4)}{\varepsilon^2}
+\frac{3F(\|\Lambda^\frac{5}{4}\mathbf{v}^\varepsilon\|_{L^2}^4)f^{-3}(\|\Lambda^\frac{5}{4}\mathbf{v}^\varepsilon\|_{L^2}^4)}{\varepsilon}\right)
(\Lambda^\frac{5}{4}G, \Lambda^\frac{5}{4}\mathbf{v}^\varepsilon)^2dt.
\end{align}
A simple calculation gives
\begin{align}
&-f^{-3}(\|\Lambda^\frac{5}{4}\mathbf{v}^\varepsilon\|_{L^2}^4)\|\Lambda^\frac{5}{4}\mathbf{v}^\varepsilon\|_{L^2}^4\nonumber\\
&=-\frac{\|\Lambda^\frac{5}{4}\mathbf{v}^\varepsilon\|_{L^2}^4}{(1+\|\Lambda^\frac{5}{4}\mathbf{v}^\varepsilon\|_{L^2}^4)^\frac{3}{4}}\nonumber\\
&=-\frac{\|\Lambda^\frac{5}{4}\mathbf{v}^\varepsilon\|_{L^2}^4+1-1}{(1+\|\Lambda^\frac{5}{4}\mathbf{v}^\varepsilon\|_{L^2}^4)^\frac{3}{4}}\nonumber\\
&= -f(\|\Lambda^\frac{5}{4}\mathbf{v}^\varepsilon\|_{L^2}^4)+f^{-3}(\|\Lambda^\frac{5}{4}\mathbf{v}^\varepsilon\|_{L^2}^4),
\end{align}
and by the definition of $f$
\begin{align}\label{5.18}
f^{-3}(\|\Lambda^\frac{5}{4}\mathbf{v}^\varepsilon\|_{L^2}^4)\|\Lambda^\frac{5}{4}\mathbf{v}^\varepsilon\|_{L^2}^2\leq 1, ~~~~f^{-2}(\|\Lambda^\frac{5}{4}\mathbf{v}^\varepsilon\|_{L^2}^4)\|\Lambda^\frac{5}{4}\mathbf{v}^\varepsilon\|_{L^2}^2\leq 1.
\end{align}
Considering \eqref{5.9*} and \eqref{5.15}-\eqref{5.18}, we have
\begin{align}\label{5.19}
dF(\|\Lambda^\frac{5}{4}\mathbf{v}^\varepsilon\|_{L^2}^4)&\leq -\frac{F(\|\Lambda^\frac{5}{4}\mathbf{v}^\varepsilon\|_{L^2}^4)f^{-3}(\|\Lambda^\frac{5}{4}\mathbf{v}^\varepsilon\|_{L^2}^4)}{\varepsilon}
\|\Lambda^\frac{5}{4}\mathbf{v}^\varepsilon\|_{L^2}^4dt\nonumber\\
&\quad+\frac{F(\|\Lambda^\frac{5}{4}\mathbf{v}^\varepsilon\|_{L^2}^4)f^{-3}(\|\Lambda^\frac{5}{4}\mathbf{v}^\varepsilon\|_{L^2}^4)}{\varepsilon}
\sqrt{\varepsilon}\|\Lambda^\frac{5}{4}\mathbf{v}^\varepsilon\|_{L^2}^2(\Lambda^\frac{5}{4}G, \Lambda^\frac{5}{4}\mathbf{v}^\varepsilon)dW\nonumber\\
&\quad+2F(\|\Lambda^\frac{5}{4}\mathbf{v}^\varepsilon\|_{L^2}^4)f^{-3}(\|\Lambda^\frac{5}{4}\mathbf{v}^\varepsilon\|_{L^2}^4)
\|\Lambda^\frac{5}{4}\mathbf{v}^\varepsilon\|_{L^2}^2\|G\|^2_{L_2(H; H^\frac{5}{4})}dt\nonumber\\
&\quad+\frac{F(\|\Lambda^\frac{5}{4}\mathbf{v}^\varepsilon\|_{L^2}^4)
f^{-2}(\|\Lambda^\frac{5}{4}\mathbf{v}^\varepsilon\|_{L^2}^4)}{2\varepsilon}\|\Lambda^\frac{5}{4}\mathbf{v}^\varepsilon\|_{L^2}^2\|G\|^2_{L_2(H; H^\frac{5}{4})}\nonumber\\
&\leq \frac{F(\|\Lambda^\frac{5}{4}\mathbf{v}^\varepsilon\|_{L^2}^4)}{\varepsilon}\left(-f(\|\Lambda^\frac{5}{4}\mathbf{v}^\varepsilon\|_{L^2}^4)+1\right)dt\nonumber\\
&\quad+2F(\|\Lambda^\frac{5}{4}\mathbf{v}^\varepsilon\|_{L^2}^4)\left(\frac{\|G\|^2_{L_2(H; H^\frac{5}{4})}}{\varepsilon}+\|G\|^2_{L_2(H; H^\frac{5}{4})}\right)dt\nonumber\\
&\quad+\frac{F(\|\Lambda^\frac{5}{4}\mathbf{v}^\varepsilon\|_{L^2}^4)f^{-3}(\|\Lambda^\frac{5}{4}\mathbf{v}^\varepsilon\|_{L^2}^4)}{\varepsilon}
\|\Lambda^\frac{5}{4}\mathbf{v}^\varepsilon\|_{L^2}^2(\Lambda^\frac{5}{4}G, \Lambda^\frac{5}{4}\mathbf{v}^\varepsilon)dW.
\end{align}

Define a stopping time
$$\tau_K=\inf\left\{t: \|\Lambda^\frac{5}{4}\mathbf{v}^\varepsilon\|_{L^2}\geq K\right\}.$$
Note that the stopping time $\tau_K$ is increasing with $\lim_{K\nearrow \infty}\tau_K=T$. Then, from \eqref{5.19} we have
\begin{align}\label{5.20}
&\mathrm{E}F(\|\Lambda^\frac{5}{4}\mathbf{v}^\varepsilon(t\wedge \tau_K)\|_{L^2}^4)\leq{\rm exp}\left(\frac{1}{\varepsilon}\right)\nonumber\\
&+C\int_{0}^{t\wedge \tau_K}\mathrm{E}F(\|\Lambda^\frac{5}{4}\mathbf{v}^\varepsilon(s)\|_{L^2}^4)
\left(-\frac{f(\|\Lambda^\frac{5}{4}\mathbf{v}^\varepsilon\|_{L^2}^4)}{\varepsilon}+\frac{1+\|G\|^2_{L_2(H; H^\frac{5}{4})}}{\varepsilon}+\|G\|^2_{L_2(H; H^\frac{5}{4})}\right)ds,
\end{align}
where the martingale part vanishes after the cut-off.
 Applying inequality $e^x(a-x)\leq e^{a-1}$, from \eqref{5.20} we have
\begin{align*}
\mathrm{E}F(\|\Lambda^\frac{5}{4}\mathbf{v}^\varepsilon(t\wedge \tau_K)\|_{L^2}^4)\leq{\rm exp}\left(\frac{1}{\varepsilon}\right)+t{\rm exp}\left(\frac{1+\|G\|^2_{L_2(H; H^\frac{5}{4})}}{\varepsilon}+\|G\|^2_{L_2(H; H^\frac{5}{4})}\right).
\end{align*}
Passing $K\rightarrow\infty$, the monotone theorem implies
\begin{align}\label{1.45}
\mathrm{E}F(\|\Lambda^\frac{5}{4}\mathbf{v}^\varepsilon(t)\|_{L^2}^4)\leq{\rm exp}\left(\frac{1}{\varepsilon}\right)+t{\rm exp}\left(\frac{1+\|G\|^2_{L_2(H; H^\frac{5}{4})}}{\varepsilon}+\|G\|^2_{L_2(H; H^\frac{5}{4})}\right).
\end{align}

Since the function $h(x)={\rm exp}\left(\frac{\left(1+x\right)^\frac{1}{4}}{\varepsilon}\right)$  is convex when $\varepsilon\leq \frac{1}{3}$, using the convexity and the Chebyshev inequality, \eqref{1.45}, we conclude
\begin{align*}
&\mathrm{P}\left(\int_{0}^{T}\|\mathbf{v}\|_{H^\frac{5}{4}}^4dt\geq R^4\right)\nonumber\\&=\mathrm{P}\left(\frac{1}{T}\int_{0}^{T}\|\mathbf{v}\|_{H^\frac{5}{4}}^4dt\geq \frac{R^4}{T}\right)\nonumber\\
&=\mathrm{P}\left(h\left(\frac{1}{T}\int_{0}^{T}\|\mathbf{v}\|_{H^\frac{5}{4}}^4dt\right)\geq h\left(\frac{R^4}{T}\right)\right)\nonumber\\
&\leq \mathrm{P}\left(\frac{1}{T}\int_{0}^{T}F(\|\mathbf{v}\|_{H^\frac{5}{4}}^4)dt\geq {\rm exp}\left(\frac{\left(1+\frac{R^4}{T}\right)^\frac{1}{4}}{\varepsilon}\right)\right)\nonumber\\
&\leq{\rm exp}\left(-\frac{\left(1+\frac{R^4}{T}\right)^\frac{1}{4}}{\varepsilon}\right)\frac{1}{T}\int_{0}^{T}\mathrm{E}F(\|\mathbf{v}\|_{H^\frac{5}{4}}^4)dt\nonumber\\
&\leq {\rm exp}\left(-\frac{\left(1+\frac{R^4}{T}\right)^\frac{1}{4}}{\varepsilon}\right)\left({\rm exp}\left(\frac{1}{\varepsilon}\right)+T{\rm exp}\left(\frac{1+\|G\|^2_{L_2(H; H^\frac{5}{4})}}{\varepsilon}+\|G\|^2_{L_2(H; H^\frac{5}{4})}\right)\right)\nonumber\\
&\leq  {\rm exp}\left(-\frac{\left(1+\frac{R^4}{T}\right)^\frac{1}{4}}{\varepsilon}\right)\left({\rm exp}\left(\frac{1}{\varepsilon}\right)+T{\rm exp}\left(\frac{1+2\|G\|^2_{L_2(H; H^\frac{5}{4})}}{\varepsilon}\right)\right).
\end{align*}
Finally, we could find $R$ large enough and $\varepsilon$ small enough such that
\begin{align*}
\mathrm{P}\left(\int_{0}^{T}\|\mathbf{v}\|_{H^\frac{5}{4}}^4dt\right)\geq {\rm exp}\left(-\frac{s}{\varepsilon}\right).
\end{align*}
We complete the proof.
\end{proof}
Finally, we have the following result.

\begin{theorem}  Let $\mathbf{u}^\varepsilon$ be the solutions of equations \eqref{e1*}, then the law of $\mathbf{u}^\varepsilon$ satisfies the uniform large deviation principle in $C([0,T]; H)$ in the sense of definition \ref{def3.1}, uniformly with respect to $x\in B_H(R)$, with speed $\varepsilon$ and with good rate function
 \begin{align}I^x(\mathbf{u})=
 \left\{\begin{array}{ll}
\!\!\frac{1}{2}\int_{0}^{T}\|\mathcal{M}(\mathbf{u})\|_{H}^2dt, ~{\rm for}~ \mathcal{M}(\mathbf{u})\in L^2(0,T; H),\\
\!\!+\infty,~  {\rm otherwise},\\
\end{array}\right.
\end{align}
where the mapping
$$\mathcal{M}:\mathbf{u}\rightarrow \mathbf{u}'+A^\alpha\mathbf{u}+B(\mathbf{u},\mathbf{u}),$$
for any $\mathbf{u}\in L^2(0,T; H^\frac{5}{4})\cap W^{1,2}(0,T; H^{-\frac{5}{4}})$.

\end{theorem}
\begin{proof} Choosing $F=C([0,T]; H^\frac{5}{4})$, $E=L^4(0,T; H^\frac{5}{4})$ and $\mathcal{E}=C([0,T]; H)$ in Proposition \ref{pro3.1}, and Lemma \ref{lem3.3}, Lemma \ref{lem5.2} implies the map
${\rm I}+\mathfrak{M}: F\rightarrow \mathcal{E}$ is Lipschitz on the ball of $E$, then based on the Proposition \ref{pro3.1}, we infer that the family of $\mathbf{u}^\varepsilon$ satisfies the uniform large deviation principle in $C([0,T]; H)$, uniformly with respect to $x\in B_H(R)$, with rate function
$$I^x(\mathbf{u})=\inf\left\{J(\mathbf{v}): \mathbf{u}=\mathbf{v}+\mathfrak{M}(\mathbf{v}),\mathbf{v}\in C([0,T]; H^\frac{5}{4})\right\},$$
where
$$J(\mathbf{v})=\frac{1}{2}\int_{0}^{T}\|\mathcal{M}(\mathbf{u})\|_{H}^2dt,$$
where the mapping
$$\mathcal{M}:\mathbf{u}\rightarrow \mathbf{u}'+A^\alpha\mathbf{u}+B(\mathbf{u},\mathbf{u}),$$
for any $\mathbf{u}\in L^2(0,T; H^\frac{5}{4})\cap W^{1,2}(0,T; H^{-\frac{5}{4}})$.
This completes the proof.
\end{proof}

\section{Large deviations of invariant measure}
In this section, we shall prove the large deviation principle of invariant measure in $H$ for 3D stochastic hyperdissipative Navier-Stokes equations.

We first introduce the action functional corresponding to equations \eqref{e1*}
$$I(\mathbf{u}):=\frac{1}{2}\int_{0}^{T}\|\mathcal{M}(\mathbf{u})\|_{H}^2dt$$
and denote $I^y(\mathbf{u})$ by
 \begin{align*}I^y(\mathbf{u})=
 \left\{\begin{array}{ll}
\!\!I(\mathbf{u}), ~{\rm if}~\mathbf{u}(0)=y,\\
\!\!+\infty,~  {\rm otherwise}.\\
\end{array}\right.
\end{align*}
Then, define the quasi-potential $U: H\rightarrow\infty$ by
$$U(x)=\inf\left\{I(\mathbf{u}): \mathbf{u}\in C([0,T];H), \mathbf{u}(0)=0, \mathbf{u}(T)=x\right\}.$$
$U(x)$ gives the minimum energy of all paths starting from $0$ to reach $x$ for any $x\in H$. For the 2D equations on a torus, using the orthogonality of $A\mathbf{u}$ and $B(\mathbf{u})$, \cite{bcm} gave the  explicit formula of the quasi-potential
\begin{eqnarray*}
U(x)=\left\{\begin{array}{ll}
\!\!\!\|x\|_{V}, x\in V\\
\!\!\!+\infty,~ x\in H \backslash V,\\
\end{array}\right.
\end{eqnarray*}
while we can not give the specific formula in 3D case.

\begin{definition}\label{def4.1} We say a family of invariant measures $\{\mu^\varepsilon\}_{\varepsilon>0}$ on $H$ satisfies the large deviation principle with the speed $\varepsilon$, and the rate function $U: H\rightarrow \infty$, if the followings hold:

i. for each $s>0$, the level set
$$K(s):=\{h\in H: ~U(h)\leq s\}$$
is compact in $H$;

ii. for any $h\in H$, for any $s\geq 0$, $\delta>0$ and $\gamma>0$, there exists $\mathcal{\varepsilon}_0$ such that for all $\varepsilon\leq \varepsilon_0$
$$\mu^x_\varepsilon(B_H(h, \delta))-{\rm exp}\left(-\frac{U(h)+\gamma}{\varepsilon}\right)\geq 0;$$

iii. for any $s\geq 0$, $\delta>0$ and $\gamma>0$, there exists $\mathcal{\varepsilon}_0$ such that for all $\varepsilon\leq \varepsilon_0$
$$\mu^x_\varepsilon(B^c_H(K(s), \delta))\leq {\rm exp}\left(-\frac{s-\gamma}{\varepsilon}\right),$$
where $B^c_H(K(s), \delta)=\{h\in H: dist_H(h, K(s))\geq \delta\}.$
\end{definition}

We state the main result of this paper.
\begin{theorem}\label{thm6.1} The family of invariant measures $\{\mu^\varepsilon\}_{\varepsilon>0}$ satisfies the large deviations in $H$ with speed $\varepsilon$ and rate function $U(x)$ in the sense of definition \ref{def4.1}.
\end{theorem}

Theorem 4.4 of \cite{bcm} proved that the level set is compact, therefore, in what follows we only verify the large deviations upper bound and lower bound following ideas from \cite{bc,sow}. For reader's convenience, we still give the details below.

\subsection{Lower bound} In this subsection, we establish the large deviations lower bound. The following exponential tightness of the invariant measure set is a fundamental result.

\begin{lemma}\label{lem3.4} For any $s>0$, there exists $\varepsilon_0$ and $R_s>0$ such that for any $\varepsilon\leq \varepsilon_0$
\begin{align*}
\mu^\varepsilon(B_H^c(0, R_s))\leq {\rm exp}\left(-\frac{s}{\varepsilon}\right).
\end{align*}
\end{lemma}
\begin{proof}Define the function
$$F(t, x)={\rm exp}\left(t+\frac{x}{\varepsilon}\right).$$
Then, using the It\^{o} formula to the function $F(t, \|\mathbf{u}^\varepsilon\|_{H}^2)$, as Lemma \ref{lem2.1*}, we could obtain
\begin{align*}
\mathrm{E}{\rm exp}\left(t+\frac{\|\mathbf{u}^\varepsilon\|_{H}^2}{\varepsilon}\right)\leq {\rm exp}\left(\frac{\|\mathbf{u}_0^\varepsilon\|_{H}^2}{\varepsilon}\right)+\int_{0}^{t}{\rm exp}(s+C\|G\|^2_{L_2(H;H)})ds,
\end{align*}
then, since $te^{-t}\leq 1$, we have
\begin{align}\label{6.7}
\mathrm{E}{\rm exp}\left(\frac{\|\mathbf{u}^\varepsilon\|_{H}^2}{\varepsilon}\right)\leq {\rm exp}\left(-t+\frac{\|\mathbf{u}_0^\varepsilon\|_{H}^2}{\varepsilon}\right)+{\rm exp}(C\|G\|^2_{L_2(H;H)}).
\end{align}
Using the ergodicity of the family of invariant measures $\mu^\varepsilon$, the Chebyshev inequality and \eqref{6.7}
\begin{align*}
\mu^\varepsilon(B_H^c(0, R_s))&=\lim_{T\rightarrow\infty}\frac{1}{T}\int_{0}^{T}\mathrm{P}(\|\mathbf{u}^\varepsilon\|_{H}\geq R_s)dt\nonumber\\
&=\lim_{T\rightarrow\infty}\frac{1}{T}\int_{0}^{T}\mathrm{P}\left(\frac{\|\mathbf{u}^\varepsilon\|^2_{H}}{\varepsilon}\geq \frac{R_s^2}{\varepsilon}\right)dt\nonumber\\
&=\lim_{T\rightarrow\infty}\frac{1}{T}\int_{0}^{T}\mathrm{P}\left({\rm exp}\left(\frac{\|\mathbf{u}^\varepsilon\|^2_{H}}{\varepsilon}\right)\geq {\rm exp}\left(\frac{R_s^2}{\varepsilon}\right)\right)dt\nonumber\\
&\leq {\rm exp}\left(-\frac{R_s^2}{\varepsilon}\right)\lim_{T\rightarrow\infty}\frac{1}{T}\int_{0}^{T}\mathrm{E}{\rm exp}\left(\frac{\|\mathbf{u}^\varepsilon\|^2_{H}}{\varepsilon}\right)dt\nonumber\\
&\leq {\rm exp}\left(-\frac{R_s^2}{\varepsilon}\right)\lim_{T\rightarrow\infty}\frac{1}{T}\int_{0}^{T}{\rm exp}\left(-t+\frac{\|\mathbf{u}_0^\varepsilon\|_{H}^2}{\varepsilon}\right)+{\rm exp}(C\|G\|^2_{L_2(H;H)})dt\nonumber\\
&\leq  {\rm exp}\left(-\frac{R_s^2-\|\mathbf{u}_0^\varepsilon\|_{H}^2-C\varepsilon\|G\|^2_{L_2(H;H)}}{\varepsilon}\right),
\end{align*}
this completes the proof by choosing $R_s$ largely.
\end{proof}

\begin{proposition}\label{pro6.1} The family of invariant measures $\mu^\varepsilon$ satisfies the large deviations lower bound in $H$ with rate function $U(x)$.
\end{proposition}
\begin{proof} Applying the continuity dependence of initial data of solution to equations \eqref{e1} and the decay estimate \eqref{3.3}, we claim that for any $x\in H$, we could find $T>0$, $\varphi_0\in L^2(0,T; H)$ and $\mathbf{u}_y\in C([0,T];H)$ such that
\begin{align}\label{4.6*}
\frac{1}{2}\int_{0}^{T}\|\varphi_0\|^2_{H}dt\leq U(x)+\frac{\gamma}{4},
\end{align}
and
\begin{align}\label{4.6}\sup_{\|y\|_{H}\leq R}\|\mathbf{u}_y(T)-x\|_{H}\leq \frac{\delta}{2},\end{align}
where $\mathbf{u}_y$ is the solution of equations
$$d\mathbf{u}+A^\frac{5}{4}\mathbf{u}dt+B(\mathbf{u},\mathbf{u})dt=G\varphi_0 dt,~ \mathbf{u}(0)=y,$$
see also \cite{bc} for more details. Since the family of $\mathbf{u}^\varepsilon$ satisfies the uniform large deviation principle in $C([0,T]; H)$, therefore, according to lower bound of uniform large deviations and \eqref{4.6*}, we have there exists $\varepsilon_0>0$ such that for any $\varepsilon\leq \varepsilon_0$
\begin{align}\label{4.7}
\inf_{y\in B_H(0,R)}\mathrm{P}\left\{\|\mathbf{u}^\varepsilon-\mathbf{u}_y\|_{C([0,T];H)}<\frac{\delta}{2}\right\}\geq {\rm exp}\left(-\frac{U(x)+\frac{\gamma}{2}}{\varepsilon}\right).
\end{align}
We infer from the invariance of measures $\mu^\varepsilon$ and \eqref{4.6} that
\begin{align}\mu^\varepsilon(B_H(x, \delta))&=\int_{H}\mathrm{P}\left\{\|\mathbf{u}^\varepsilon(T)-x\|_{H}\leq \delta\right\}d\mu^\varepsilon\nonumber\\
&\geq \int_{H}\mathrm{P}\left\{\|\mathbf{u}^\varepsilon(T)-\mathbf{u}_y\|_{H}\leq \frac{\delta}{2}\right\}d\mu^\varepsilon\nonumber\\
&\geq \int_{H}\mathrm{P}\left\{\|\mathbf{u}^\varepsilon(T)-\mathbf{u}_y\|_{C([0,T];H)}\leq \frac{\delta}{2}\right\}d\mu^\varepsilon\nonumber\\
&\geq \int_{B_H(0,R)}\mathrm{P}\left\{\|\mathbf{u}^\varepsilon(T)-\mathbf{u}_y\|_{C([0,T];H)}\leq \frac{\delta}{2}\right\}d\mu^\varepsilon\nonumber\\
&\geq \mu^\varepsilon(B_H(0,R))\inf_{y\in B_H(0,R)}\mathrm{P}\left\{\|\mathbf{u}^\varepsilon(T)-\mathbf{u}_y\|_{C([0,T];H)}\leq \frac{\delta}{2}\right\}.
\end{align}
By Lemma \ref{lem3.4}, we could find $\varepsilon_0>0$ such that for $\varepsilon\leq \varepsilon_0$
\begin{align}\label{4.9}
\mu^\varepsilon(B_H(0,R))\geq 1-{\rm exp}\left(-\frac{s}{\varepsilon}\right)\geq 1-{\rm exp}\left(-\frac{s}{\varepsilon_0}\right)\geq \frac{1}{2}.
\end{align}
From inequalities \eqref{4.7}-\eqref{4.9}, we have
\begin{align*}
\mu^\varepsilon(B_H(x, \delta))\geq \frac{1}{2} {\rm exp}\left(-\frac{U(x)+\frac{\gamma}{2}}{\varepsilon}\right).
\end{align*}
This completes the proof of lower bound.
\end{proof}

\subsection{Upper bound}

In this subsection, we establish the large deviations upper bound.
\begin{proposition}\label{pro4.2*}
  The family of invariant measures $\mu^\varepsilon$ satisfies the large deviations upper bound in $H$ with rate function $U(x)$.
\end{proposition}

We first introduce the two preliminaries lemmas.
\begin{lemma}\label{lem4.1} For any $\delta>0$ and $s>0$, there exist $\lambda>0$ and $T>0$ such that for any $t\geq T$ and $\mathbf{u}\in C([0,t];H)$, if
$$\|\mathbf{u}(0)\|_{H}<\lambda,~ I(\mathbf{u})\leq s,$$
then,
$${\rm dist}_H(\mathbf{u}(t), K(s))<\delta,$$
where $K(s)$ is the set in definition \ref{def4.1}.
\end{lemma}
\begin{proof} The proofs rely on the decay estimate in Lemma \ref{lem3.2*} and regularity estimate  Lemma \ref{lem3.1}\eqref{3.3}, see \cite{sow} and \cite[Theorem 7.2]{bc} for more details.
\end{proof}

\begin{lemma} \cite[Theorem 7.3]{bc} or \cite{sow} For any $s>0$, $\delta>0$ and $r>0$, let $\lambda$ be as in Lemma \ref{lem4.1}. There exists $N\in \mathbb{N}$ such that for $\mathbf{u}\in H_{r,s,\delta}(N)$, where the set $H_{r,s,\delta}(N)$ is defined by
$$H_{r,s,\delta}(N)=\left\{\mathbf{u}\in C([0,N]; H):~ \|\mathbf{u}(0)\|_{H}\leq r, ~\|\mathbf{u}(j)\|_{H}\geq \lambda,~j=1,2,\cdots, N\right\}.$$
Then, we have
$$\inf\left\{I(\mathbf{u}): \mathbf{u}\in H_{r,s,\delta}(N)\right\}>s.$$
\end{lemma}

{\bf Proof of Proposition \ref{pro4.2*}}. Using the invariance of the family of measures $\mu^\varepsilon$,
\begin{align}\label{6.8}
&\mu^\varepsilon\{h\in H: dist_H(h, K(s))\geq \delta\}\nonumber\\
&=\int_{H}\mathrm{P}\{ dist_H(\mathbf{u}^\varepsilon, K(s))\geq \delta\}d\mu^\varepsilon\nonumber\\
&=\int_{B_{H}(0,R_s)}\mathrm{P}\{ dist_H(\mathbf{u}^\varepsilon, K(s))\geq \delta\}d\mu^\varepsilon\nonumber\\
&\quad+\int_{B^c_{H}(0,R_s)}\mathrm{P}\{ dist_H(\mathbf{u}^\varepsilon, K(s))\geq \delta\}d\mu^\varepsilon\nonumber\\
&=\int_{B_{H}(0,R_s)}\mathrm{P}\{ dist_H(\mathbf{u}^\varepsilon, K(s))\geq \delta, \mathbf{u}^\varepsilon\in H_{r,s,\delta}(N)\}d\mu^\varepsilon\nonumber\\
&\quad+\int_{B_{H}(0,R_s)}\mathrm{P}\{ dist_H(\mathbf{u}^\varepsilon, K(s))\geq \delta, \mathbf{u}^\varepsilon\notin H_{r,s,\delta}(N)\}d\mu^\varepsilon\nonumber\\
&\quad+\int_{B^c_{H}(0,R_s)}\mathrm{P}\{dist_H(\mathbf{u}^\varepsilon, K(s))\geq \delta\}d\mu^\varepsilon.
\end{align}
Lemma \ref{lem3.4} implies there exists $\varepsilon_0>0$ such that for any $\varepsilon<\varepsilon_0$
\begin{align}\label{6.9}
\int_{B^c_{H}(0,R_s)}\mathrm{P}\{ dist_H(\mathbf{u}^\varepsilon, K(s))\geq \delta\}d\mu^\varepsilon\leq \mu^\varepsilon(B_H^c(0, R_s))\leq {\rm exp}\left(-\frac{s}{\varepsilon}\right).
\end{align}
Note that $H_{r,s,\delta}(N)$ is a closed set in $C([0,N];H)$, then from the uniform large deviations upper bound of $\mathbf{u}^\varepsilon$ we could deduce that
\begin{align}\label{6.10}
&\int_{B_{H}(0,R_s)}\mathrm{P}\{ dist_H(\mathbf{u}^\varepsilon, K(s))\geq \delta, \mathbf{u}^\varepsilon\in H_{r,s,\delta}(N)\}d\mu^\varepsilon\nonumber\\
&\leq\int_{B_{H}(0,R_s)}\mathrm{P}\{ \mathbf{u}^\varepsilon\in H_{r,s,\delta}(N)\}d\mu^\varepsilon\nonumber\\
&\leq \mu^\varepsilon(B_H(0, R_s))\sup_{y\in B_{H}(0,R_s)}\mathrm{P}\{ \mathbf{u}^\varepsilon\in H_{r,s,\delta}(N)\}\nonumber\\
&\leq {\rm exp}\left(-\frac{s-\gamma/2}{\varepsilon}\right).
\end{align}
It remains to show that
\begin{align}
\int_{B_{H}(0,R_s)}\mathrm{P}\{ dist_H(\mathbf{u}^\varepsilon, K(s)), \mathbf{u}^\varepsilon\notin H_{r,s,\delta}(N)\}d\mu^\varepsilon\leq  {\rm exp}\left(-\frac{s-\gamma/2}{\varepsilon}\right).
\end{align}
First, using the Markov property of $\mathbf{u}^\varepsilon$, we denote by $\mathbb{P}(\tau,t,dr)$ the transition semigroup, then
\begin{align}\label{6.12}
&\int_{B_{H}(0,R_s)}\mathrm{P}\{dist_H(\mathbf{u}^\varepsilon, K(s))\geq \delta, \mathbf{u}^\varepsilon\notin H_{r,s,\delta}(N)\}d\mu^\varepsilon\nonumber\\
&\leq \int_{B_{H}(0,R_s)}\mathrm{P}\{dist_H(\mathbf{u}^\varepsilon, K(s))\geq \delta, \|\mathbf{u}^\varepsilon(j)\|_{H}\leq \lambda, j=1,2,\cdots,N\}d\mu^\varepsilon\nonumber\\
&\leq \sum_{j=1}^N\int_{B_{H}(0,R_s)}\mathrm{P}\{ dist_H(\mathbf{u}^\varepsilon, K(s))\geq \delta, \|\mathbf{u}^\varepsilon(j)\|_{H}\leq \lambda\}d\mu^\varepsilon\nonumber\\
&\leq \sum_{j=1}^N\int_{B_{H}(0,R_s)}\mathrm{P}\{dist_H(\mathbf{u}^\varepsilon, K(s))\geq \delta, \|\mathbf{u}^\varepsilon(j)\|_{H}\leq \lambda\}d\mu^\varepsilon\nonumber\\
&=\sum_{j=1}^N\int_{B_{H}(0,R_s)}\int_{\|\mathbf{u}^\varepsilon(j)\|_{H}\leq \lambda}\mathrm{P}\{ dist_H(\mathbf{u}_r^\varepsilon(t-j), K(s))\geq \delta, \}\mathbb{P}(\tau,y,dr)d\mu^\varepsilon\nonumber\\
&\leq \sum_{j=1}^N\int_{B_{H}(0,R_s)}\sup_{\|r\|_{H}<\lambda}\mathrm{P}\{dist_H(\mathbf{u}_r^\varepsilon(t-j), K(s))\geq \delta\}d\mu^\varepsilon\nonumber\\
&\leq \sum_{j=1}^N\sup_{\|r\|_{H}<\lambda}\mathrm{P}\{ dist_H(\mathbf{u}_r^\varepsilon(t-j), K(s))\geq \delta\}.
\end{align}
In order to use the large deviations upper bound of $\mathbf{u}^\varepsilon$ in $C([0,t]; H)$, we introduce the auxiliary sets to build the bridge between the spaces $H$ and $C([0,t]; H)$. According to Lemma \ref{lem4.1}, define the sets by
$$\mathcal{S}_1=\{\mathbf{u}^\varepsilon\in C([0,t];H): I(\mathbf{u}^\varepsilon)\leq s, |\mathbf{u}^\varepsilon(0)|_{H}\leq \lambda\},$$
and
$$\mathcal{S}_2=\{\mathbf{u}^\varepsilon\in C([0,t];H): I(\mathbf{u}^\varepsilon)\leq s, \mathbf{u}^\varepsilon(0)=r\}.$$
Then, since $\|r\|_{H}<\lambda$, we have  $\mathcal{S}_2\subset\mathcal{S}_1$. Next, we show that
\begin{align}\label{1*}\left\{\omega, dist_H(\mathbf{u}^\varepsilon(t);K(s))\geq \delta\right\}\subset \left\{\omega, dist_{C([0,t]; H)}(\mathbf{u}^\varepsilon;\mathcal{S}_1)\geq \frac{\delta}{2}\right\}.\end{align}
For any $\xi\in \mathcal{S}_1$,
$$dist_H(\mathbf{u}^\varepsilon(t);K(s))\leq \|\mathbf{u}^\varepsilon-\xi\|_{H}+ dist_H(\xi(t);K(s)).$$
From Lemma \ref{lem4.1}, we know
$$dist_H(\xi(t);K(s))\leq \frac{\delta}{2}.$$
Moreover, since
$$\|\mathbf{u}^\varepsilon-\xi\|_{H}\leq \|\mathbf{u}^\varepsilon-\xi\|_{C([0,t];H)},$$
therefore, provided $\|\mathbf{u}^\varepsilon-\xi\|_{C([0,t];H)}< \frac{\delta}{2}$ for any $\xi\in \mathcal{S}_1$, then $$dist_H(\mathbf{u}^\varepsilon(t);K(s))< \delta,$$ we obtain \eqref{1*}.
Then,
\begin{align}\label{1**}\mathrm{P}\left\{\omega, dist_H(\mathbf{u}^\varepsilon(t);K(s))\geq \delta\right\}&\leq \mathrm{P}\left\{\omega, dist_{C([0,t]; H)}(\mathbf{u}^\varepsilon;\mathcal{S}_1)\geq \frac{\delta}{2}\right\}\nonumber\\
&\leq\mathrm{P}\left\{\omega, dist_{C([0,t]; H)}(\mathbf{u}^\varepsilon;\mathcal{S}_2)\geq \frac{\delta}{2}\right\} .\end{align}
Furthermore, the set
$$\left\{\mathbf{u}\in C([0,t]; H), dist_{C([0,t]; H)}(\mathbf{u}^\varepsilon;\mathcal{S}_2)\geq \frac{\delta}{2}\right\}$$
is the closed set in $C([0,t]; H)$. Then, by the upper bound estimate of uniform large deviations, we have
\begin{align*}\sup_{y\in B_H(0,\lambda)}\mathrm{P}\left\{ dist_{C([0,t]; H)}(\mathbf{u}_y^\varepsilon;\mathcal{S}_2)\geq \frac{\delta}{2}\right\}\leq {\rm exp}\left(-\frac{s-\gamma}{\varepsilon}\right),\end{align*}
which follows from \eqref{1**}
\begin{align}\label{2*}\sup_{y\in B_H(0,\lambda)}\mathrm{P}\left\{ dist_H(\mathbf{u}_y^\varepsilon(t);K(s))\geq \delta\right\}\leq {\rm exp}\left(-\frac{s-\gamma}{\varepsilon}\right).\end{align}

Finally, by \eqref{6.8}, \eqref{6.10} and \eqref{6.12}, \eqref{2*}, we infer that there exists $\varepsilon_0$ small enough such that for all $\varepsilon<\varepsilon_0$
$$\mu^x_\varepsilon(B^c_H(K(s), \delta))\leq {\rm exp}\left(-\frac{s-\gamma}{\varepsilon}\right).$$
Proposition \ref{pro4.2*} is established.

We conclude that Theorem \ref{thm6.1} holds following from Proposition \ref{pro6.1} and Proposition \ref{pro4.2*}.

\section*{Acknowledgments}
H. Liu is supported by the National Natural Science Foundation of China (Grant Nos.12271293; 11901342). C. Sun is supported by the National Natural Science Foundation of China (Grant No.11701269).

\section*{Data availability}
Data sharing not applicable to this article as no datasets were generated or analysed
during the current study.

\section*{Statements and Declarations}
The authors have no relevant financial or non-financial interests to disclose

\end{document}